\documentclass[12pt]{article}
\usepackage{a4wide} 
\usepackage{overpic}
\usepackage{amstext}
\usepackage{amsmath} 
\usepackage{amssymb}
\usepackage{amsthm}
\usepackage{graphicx}
\usepackage{array}
\usepackage{graphics,graphicx}        
\newtheorem{rem}{Remark}[section]
\newtheorem{prop}{Proposition}[section]
\newtheorem{thm}{Theorem}[section]

\newcommand{\bfeps}{\boldsymbol \epsilon}

\newcommand{\Div}{\text{{\bf div}$~$}}
\renewcommand{\div}{\text{\rm div$~$}}
\newcommand{\bftheta}{\boldsymbol \theta}
\newcommand{\trho}{\widetilde{\rho}}
\newcommand{\IR}{\mathbb{R}}
\newcommand{\bfn}{\boldsymbol n}
\newcommand{\bfu}{\boldsymbol u}
\newcommand{\bff}{\boldsymbol f}
\newcommand{\bfsig}{\boldsymbol{\sigma}}
\newcommand{\bfzero}{\boldsymbol 0}
\newcommand{\bfI}{\boldsymbol I}

\newcommand{\bfz}{\boldsymbol z}
\newcommand{\bfp}{\boldsymbol p}
\newcommand{\bfv}{\boldsymbol v}
\newcommand{\mcT}{\mathcal{T}}
\newcommand{\mcL}{\mathcal{L}}
\newcommand{\bfw}{\boldsymbol w}
\newcommand{\bfr}{\boldsymbol r}
\newcommand{\bfe}{\boldsymbol e}
\newcommand{\Ps}{{P}_\Gamma}
\newcommand{\bfx}{\boldsymbol x}

\title{Hybridized Augmented Lagrangian Methods for Contact Problems}
\date{\today}

\author{
Erik Burman,
Peter Hansbo,
Mats G. Larson 
}
\date{}

\begin{document}

\maketitle

\begin{abstract} 
This paper addresses the problem of friction-free contact between two elastic bodies. We develop an augmented Lagrangian method that provides computational convenience by reformulating the contact problem as a nonlinear variational equality. To achieve this, we propose a Nitsche-based method incorporating a hybrid displacement variable defined on an interstitial layer. This approach enables complete decoupling of the contact domains, with interaction occurring exclusively through the interstitial layer. The layer is independently approximated, eliminating the need to handle intersections between unrelated meshes. Additionally, the method supports introducing an independent model on the interface, which we leverage to represent a membrane covering one of the bodies.  We present the formulation of the method, establish stability and error estimates, and demonstrate its practical utility through illustrative numerical examples.
\end{abstract}


\section{Introduction}
\label{sec:1}

Traditionally, two-body contact algorithms in finite element analysis rely on a node-to-segment approach, where the nodes of one mesh are constrained from crossing the discrete boundary of the other. In a one-pass algorithm, only the nodes on one of the surfaces are considered, which can result in local penetration if the mesh densities differ significantly. A two-pass algorithm, on the other hand, considers nodes on both surfaces, but this can lead to ill-posed problems, cf. El-Abbasi and Bathe \cite{EABa01}. This arises when nodes on the two surfaces are very close, as they may impose nearly identical contact conditions. To address this issue, additional checks are necessary, as discussed by Puso and Laursen \cite{PuLa04}.

The contact constraints can be enforced using discrete Lagrange multipliers (contact pressures) associated with the nodes or a nodal penalty method that penalizes the no-penetration constraint. Another classical approach is the distributed Lagrange multiplier method \cite{EABa01,BeHiLa99}, where the multiplier's discretization is typically related to the surface mesh of one of the bodies to ensure stability. Consequently, these approaches require careful handling of intersections between unrelated meshes. This challenge also applies to distributed penalty methods and Nitsche's method.

To overcome these limitations and enable a more flexible approximation of the interface variables, we extend the hybrid Nitsche's method \cite{BuElHaLaLa19},  which incorporates an independent displacement field at the interface,  to the case of friction-free elastic contact. The hybrid field can serve as an auxiliary variable without direct physical interpretation, facilitating the transfer of information between the contacting bodies. For instance, the hybrid field may be defined on a structured mesh for computational convenience. Alternatively, the hybrid variable can represent physical phenomena, such as a membrane covering one of the bodies or a shell located between the two bodies.

We formulate this approach within an augmented Lagrangian framework, leveraging Rockafellar's reformulation \cite{Ro73, Ro73b, Rock74} of the well-known Kuhn-Tucker contact conditions. This reformulation was first proposed and analyzed by Chouly and Hild \cite{ChHi13} for Nitsche's method in the context of friction-free contact (see also the review by Chouly \textit{et al.} \cite{ChFaHiMlPoRe17}). This framework enables the definition of the Nitsche stabilization mechanism for variational inequality problems, transforming the inequality constraints associated with contact into nonlinear \textit{equalities} (see also \cite{BuHaLa23}). These equalities facilitate the application of iterative solution schemes in the spirit of Alart and Curnier \cite{AC91}.

Our method shares similarities with that of Chouly and Hild \cite{ChHi13} but with a significant distinction: it allows for an independent approximation of the displacement field in the contact zone, decoupled from the approximation used in the elastic bodies. We establish theoretical results, including stability and approximation properties, to support the method's robustness. Finally, we provide several numerical examples to demonstrate our method's practical application and performance.

The structure of the paper is as follows: Section \ref{sec:contact} introduces the contact model, derives the hybridized finite element method, and formulates the associated optimality equations.  Section \ref{se:error-estimates} presents a stability estimate, a best approximation result, and a discussion of the method's convergence order.
Section \ref{sec:numerics} provides several numerical examples to illustrate the practical application and performance of the method. Section \ref{sec:conclusions} concludes the paper with a summary of key findings and insights.

\section{A Contact Model Problem}\label{sec:contact}

\subsection{Hybridized Problem Formulation}

We shall study several different contact problems between two elastic bodies occupying the domains $\Omega_i \subset \IR^d$, $i=1,2$, with $d = 2$ or $3$. Let $\Omega_0$ denote a hybrid object located between $\Omega_1$ and $\Omega_2$, see Fig. \ref{fig:domains}. For $d = 3$, the hybrid object is a surface, and for $d = 2$, it is a curve. The bodies $\Omega_1$ and $\Omega_2$ do not come into direct contact; all interactions are mediated through $\Omega_0$. Thus, all information exchanged between $\Omega_1$ and $\Omega_2$ passes through the hybrid space.  We will consider two scenarios:

\begin{itemize} \item {\bf Case 1.} The hybrid space is an auxiliary tool to facilitate numerical computations. For example, it may consist of a structured mesh that efficiently transfers data between two unstructured meshes. In this scenario, the hybrid space must be stabilized, potentially by ensuring it conforms to the boundary of one of the bodies.

\item {\bf Case 2.} The hybrid space represents physical interactions between $\Omega_1$ and $\Omega_2$, such as a membrane or a plate residing between them, as described in \cite{HaLa22}. \end{itemize}

In both cases, the hybrid object $\Omega_0$ is treated as the master object, while the elastic bodies $\Omega_i$, $i=1,2$, implement contact or equality constraints. The hybrid object is only influenced by the normal stresses exerted by the two elastic bodies.

We will demonstrate that both cases can be conveniently addressed and analyzed within a unified abstract framework. To achieve this, we will not initially specify the exact properties of the forms associated with the hybrid object. Instead, we will develop the analysis based on general abstract assumptions and define the specific hybrid forms for the applications presented below. This approach emphasizes the flexibility and generality of the framework.
\begin{figure}
  \centering
\includegraphics[width=3in]{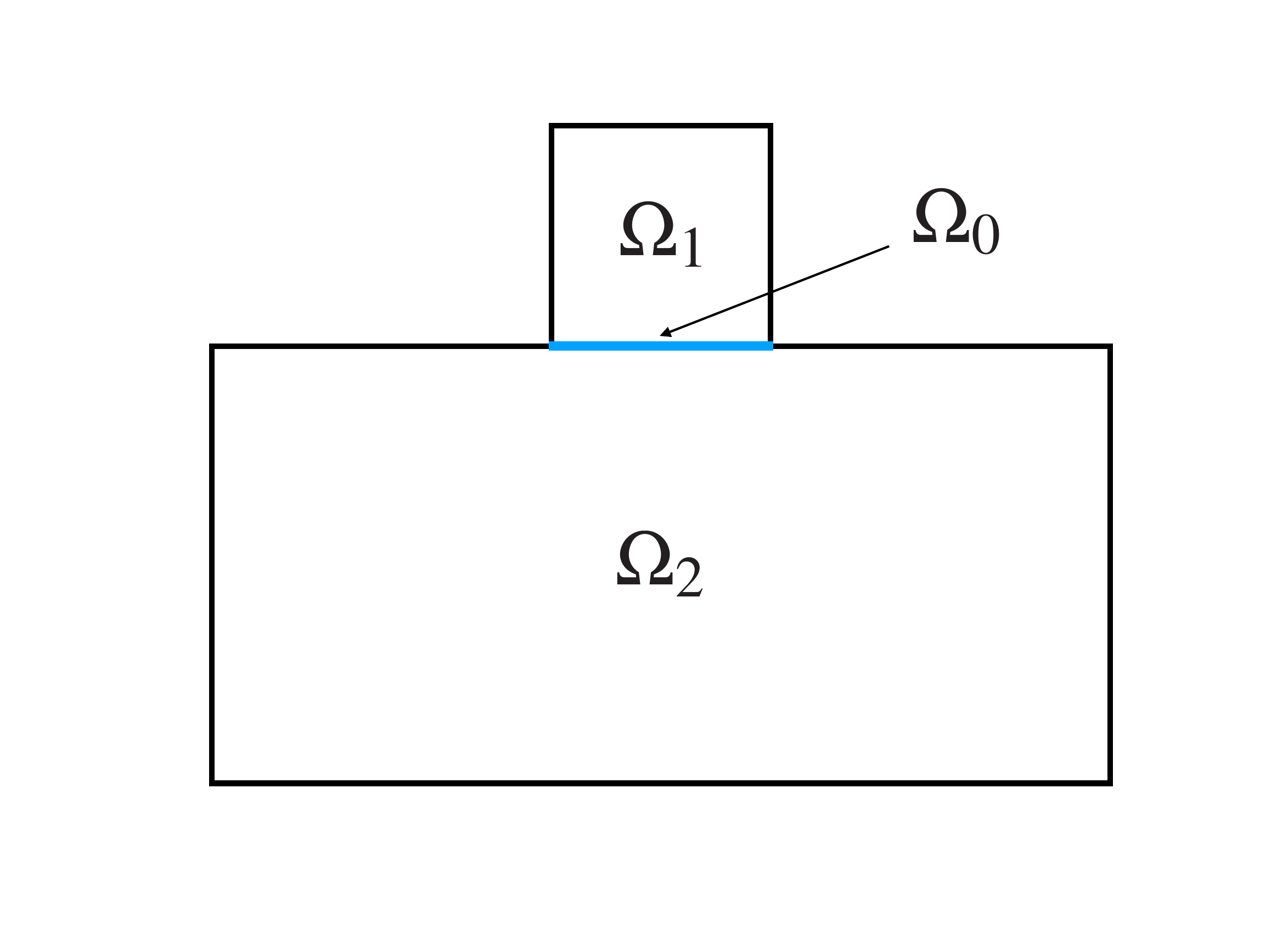}
    \caption{Example of the different domains $\Omega_i$.}
  \label{fig:domains}
\end{figure}
\

\paragraph{Governing Equations for the Elastic Bodies.} Following the problem definition of Fabre, Pousin, and Renard \cite{FaPoRe16}, the boundary $\partial \Omega_i$ of $\Omega_i$, for 
$i=1,2,$ is divided into three nonoverlapping parts: $\partial \Omega_{i,C}$ (the potential zone of contact), $\partial \Omega_{i,N}$ (the Neumann part), and $\partial \Omega_{i,D}$ (the Dirichlet part) which we, in the analysis, assume have non-zero measure to guarantee that Korn's inequality holds. We let $\bfn_i$ be the exterior unit normal to $\partial \Omega_i$.

Let $\bfu_i:\Omega_i \rightarrow \IR^d$ denote the displacement field on $\Omega_i$. We 
assume that the two elastic bodies $\Omega_i$, $i=1,2,$
are subjected to volume forces $\bff_i$ and, for simplicity, zero displacements on $\partial \Omega_{i,D}$ and zero tractions on 
$\partial \Omega_{i,N}$, 
\begin{alignat}{3}
-\nabla\cdot\bfsig(\bfu_i)  &= \bff_i & \qquad & \text{in $\Omega_{i}$}  \label{eq:omegaibody}
 \\
  \bfu_i &=  \bfzero & \qquad &\text{on $\partial\Omega_{i,D}$}
  \label{eq:omegaidir} 
  \\
  \bfsig_n(\bfu_i) &=  \bfzero & \qquad &\text{on $\partial\Omega_{i,N}$} 
  \label{eq:omegaineu} 
\end{alignat}
We further assume that Hooke's constitutive law holds
\begin{align}
  \bfsig(\bfu_i) =  \lambda_i ~\text{tr}\,\bfeps(\bfu_i)\,\bfI 
   + 2 \mu_i \bfeps(\bfu_i), \qquad \bfeps(\bfu_i) = \frac12(\bfu_i \otimes \nabla+ \nabla \otimes \bfu_i )
\end{align}
with Lam\'e parameters $\mu_i$ and $\lambda_i$. We must add the equality or contact constraints on $\partial \Omega_{i,C}$ to complete the equations.

\paragraph{The Equality and Inequality Constraints.}
To define our contact constraints, we start by defining the distance between $\Omega_i$ and $\Omega_0$,
\begin{align}\label{eq:rhodef}
\rho_{i,0}(\bfz) =  (\bfn_{i,0}(\bfz),\bfz - \bfp_{0}(\bfz))_{\IR^d}
\end{align}
where $\bfz \in \partial \Omega_{i,C}$,  $\bfp_{0}(\bfz)$ is the closest point mapping associated with $\Omega_0$,  and $\bfn_{i,0}(\bfz) = \bfn_{i,0}(\bfp_0(\bfz))$ is the pullback of the normal to $\Omega_0$ pointing into $\Omega_i$. Taking the deformation fields into account by replacing $\bfz - \bfp_0(\bfz)$ by $(\bfz+\bfv_i(\bfz)) - (\bfp_0(\bfz) - \bfv_0(\bfp_0(\bfz))$, gives 
\begin{align}
\trho_{i,0}(\bfz,\bfv_0,\bfv_i) &=  (\bfn_{i,0}(\bfz),(\bfz+\bfv_i(\bfz))_{\IR^d} - (\bfp_0(\bfz) + \bfv_0(\bfp_0(\bfz))_{\IR^d}
\\
&=  (\bfn_{i,0}(\bfz),(\bfz- \bfp_0(\bfz))_{\IR^d}
+(\bfn_{i,0}(\bfz),\bfv_i(\bfz)-\bfv_0(\bfp_0(\bfz))_{\IR^d}
\\
&=  \rho_{i,0}(\bfz)
+(\bfn_{i,0}(\bfz),\bfv_i(\bfz)-\bfv_0(\bfp_0(\bfz))_{\IR^d}
\\
&= \rho_{i,0}(\bfz) - [v_n]_i
\end{align}
Thus, we may extend the definition (\ref{eq:rhodef}) to the deformed case by
\begin{align}\label{eq:dist-def}
\boxed{\trho_{i,0}(\bfz,\bfv_0,\bfv_i) 
=\rho_{i,0}(\bfz) - [v_n(\bfz)]_i\quad \text{on $\partial \Omega_{i,C}$}}
\end{align}
where 
\begin{align}
\boxed{[v_n(\bfz)]_i = (\bfn_{i,0}(\bfz),\bfv_i(\bfz))_{\IR^d} 
-(\bfv_0(\bfp_0(\bfz))_{\IR^d}}
\end{align}
is the jump in the normal displacements. 
In the following, we make  the dependence on $\bfz$ implicit and use the more compact notation $\trho_{i,0}(\bfv_0,\bfv_i)$ etc.

\begin{rem}
Note that the exterior unit 
normal to the elastic body satisfies $\bfn_i = - \bfn_0$ in the contact 
zone and $\bfn_i \approx - \bfn_0$ close to the contact if the boundary 
is smooth and thus
\begin{equation}
[v_n]\approx v_{i,n} + v_{0,n} = (\bfn_{i},\bfv_i)_{\IR^d}
+
(\bfn_{0}, \bfv_0)_{\IR^d}
\end{equation}
which is the standard definition of the jump in the normal displacement.
\end{rem}

We then have the following constraints,
\begin{itemize}
\item {\bf Hybridized equality constraint:} 
\begin{alignat}{3}
  \trho_{i,0}(\bfu_i,\bfu_0) &=  0  &\qquad & \text{on $\partial\Omega_{i,C}$}  \label{eq:equality}
  \end{alignat}
\item {\bf Hybridized contact constraints:} 
\begin{alignat}{3}
  -\trho_{i,0}(\bfu_i,\bfu_0) &\leq  0  &\qquad & \text{on $\partial\Omega_{i,C}$}  \label{eq:kuhn1}
  \\
  \sigma_n(\bfu_i) &\leq  0 & \qquad &\text{on $\partial\Omega_{i,C}$} \label{eq:kuhn2} 
  \\
\sigma_n(\bfu_i) \trho_i(\bfu_i,\bfu_0) &= 0 &  \qquad &\text{on $\partial\Omega_{i,C}$} \label{eq:kuhn3}
\end{alignat}
where $\sigma_n(\bfv)$ is the (scalar) normal surface stress
\begin{align}\label{eq:signn}
\sigma_n(\bfv_i) = \bfn_i \cdot \bfsig(\bfv_i)\cdot \bfn_i
\end{align}
\end{itemize}

\subsection{A Hybrid Nitsche Finite Element Method}

\paragraph{Function and Finite Element Spaces.} We first define 
the natural function spaces for our continuous problem. Let 
$V_i = H^1(\Omega_i)$ for $i=1,2,$ and $V_0 = H^{s_0}(\Omega_0)$ 
where $s_0$ depends on the physical properties of the 
hybrid space. For convenience, we define the product space
\begin{align}
W = V_0 \oplus V_1 \oplus V_2
\end{align}
of $\IR^d$ valued displacement fields. 

Next, we define the corresponding finite element spaces. To that end let $\mcT_{h,i}$ be a quasi-uniform partition of $\Omega_{i}$, into shape regular elements $T$, with mesh-parameter $h_i \in (0,h_0]$. Let $V_{h,i} \subset V_i$ be a conforming finite element space on $\mcT_{h,i}$ consisting of piecewise polynomials of order $p_i$. Let $\pi_{h,i}:H^1(\Omega_i) \rightarrow V_{h,i}$ be an interpolation operator such that 
\begin{equation}\label{eq:interpol}
\boxed{ \| v - \pi_{h,i} v \|_{H^m(T)} \lesssim h_i^{k-m} \| v \|_{H^k(N(T))}, \qquad
0\leq m \leq k \leq p_i+1 }
\end{equation}
We finally define the product of the finite element spaces
\begin{equation}
W_h = V_{h,0} \oplus V_{h,1} \oplus V_{h,2}
\end{equation}
where $V_{h,i}$ are the $d$-dimensional versions of the corresponding 
scalar spaces.

\paragraph{Augmented Lagrangian Formulation.} We consider the nonlinear augmented Lagrangian formulation of Rockafellar \cite{Ro73,Ro73b,Rock74}, introduced in contact analysis by Alart and Curnier \cite{AC91}. The basic idea is to write the Kuhn--Tucker conditions (\ref{eq:kuhn1})--(\ref{eq:kuhn3}) in the equivalent form
\begin{equation}\label{eq:rocka}
\boxed{\sigma_n(\bfu_i)
= [\sigma_n(\bfu_i) + \gamma_i \trho_{i,0}(\bfu_0,\bfu_i) ]_-
= [\sigma_n(\bfu_i) - \gamma( [u_n]_i - \rho_{i,0}) ]_-}
\end{equation}
where $[x]_- = \min(x,0)$ and $\gamma > 0$ is a parameter, see Appendix \ref{sec:kkt} and \cite{ChHi13}. Dimensional analysis indicates that 
\begin{equation}\label{eq:gamma}
\boxed{\gamma_i = \gamma_0 h_i^{-1}}
\end{equation}
for a parameter $\gamma_0$ not dependent on $h_i$, 
and we will see in the forthcoming analysis that this is indeed the proper choice. 

To handle equality and inequality constraints in the same formulation, we define
\begin{align}\label{eq:S}
\boxed{S_i(\bfv) = 
\begin{cases}
\sigma_n(\bfv_i)-\gamma_0 h_i^{-1} ([v_n]_i - \rho_{i,0}) &\text{Equality}
\\
[\sigma_n(\bfv_i)-\gamma_0 h_i^{-1} ([v_n]_i - \rho_{i,0})]_- &\text{Inequality}
\end{cases}
}
\end{align}
Note that $S_i(\bfv)$ is a function of $\bfv_0$ and $\bfv_i$, so that 
$S_i(\bfv) = S_i(\bfv_i,\bfv_0)$.

\begin{rem} The quantity 
\begin{equation}
\Sigma_{n,i}(\bfv_0,\bfv_i) = \sigma_n(\bfv_i)-\gamma_0 h_i^{-1} ([v_n]_i - \rho_{i,0})
\end{equation}
is the so-called Nitsche normal stress, which is the natural approximation of the 
normal stress provided by the method and thus
\begin{align}\label{eq:S-nitsche}
\boxed{S_i(\bfv) = 
\begin{cases}
\Sigma_{n,i}(\bfv_0,\bfv_i) 
 &\text{Equality}
\\
[\Sigma_{n,i}(\bfv_0,\bfv_i) ]_- &\text{Inequality}
\end{cases}
}
\end{align}
We note that the inequality constraint only allows negative normal stress.
\end{rem}

We can now formulate a discrete minimum problem 
\begin{align}
\boxed{\bfu_h = \text{argmin}_{\bfv_h \in W_h} \mcL_A(\bfv_h)}
\end{align}
where the augmented Lagrangian takes the form
\begin{align}
\mcL_A(\bfv) &= \sum_{i=0}^2 \frac12 a_i(\bfv,\bfv) - l_i(\bfv) 
\\
&\qquad  + \sum_{i=1}^2 \frac12 \gamma_0^{-1} \|S_i(\bfv)\|^2_{H^{-1/2}_h(\partial \Omega_{i,C})} 
- \frac12 \gamma_0^{-1} \| \sigma_n (\bfv_i)) \|^2_{H^{-1/2}_h(\partial \Omega_{i,C})} 
\end{align}
Here, for $i=1,2,$ the forms are defined by
\begin{align}
a_i(\bfv,\bfw) &= (\bfsig_i(\bfv_i),\bfeps(\bfw_i))_{\Omega_i}
 \\
 l_i(\bfv) &= (\bff_i,\bfv_i)_{\Omega_i}
\end{align}
and the forms $a_0$ and $l_0$ are related to the hybrid space and may model some physics 
(Case 2) or could be zero (Case 1).  In the latter case, $\Omega_0$ is connected weakly 
to $\Omega_{i,C}$ for one of the domains $\Omega_i$, say $i=1$, through an equality 
constraint (\ref{eq:equality}).  Here and below also use the scalar products
\begin{align}
(v,w)_{H^{s}_h(\partial \Omega_{i,C})} = h_i^{-{2s}}( v,w)_{\partial \Omega_{i,C}}, \qquad s \in \IR
\end{align}
with associated norms
\begin{align}
\| v \|^2_{H^{s}_h(\partial \Omega_{i,C})} = (v,v)_{H^{s}_h(\partial \Omega_{i,C})}
= h_i^{-2s} \| v \|^2_{\partial \Omega_{i,C}},   \qquad  s \in \IR
\end{align}
These scalar products mimic the corresponding continuous $H^{s}(\partial \Omega_{i,C})$ product on the 
discrete spaces.

\subsection{Optimality Equations}

The optimality equations take the form: find $\bfu_h \in W_h$ such that 
\begin{align}\label{eq:fem}
\boxed{A(\bfu_h, \bfv) = l(\bfv)\qquad \forall \bfv \in W_h}
\end{align}
Here, the forms are defined by 
\begin{align}\label{eq:form-A}
A(\bfv,\bfw) &=a(\bfv,\bfw) + b(\bfv,\bfw) - c(\bfv,\bfw)
\\
a(\bfv,\bfw)&= 
\sum_{i \in I_a} a_i(\bfv_i,\bfw_i) 
\\
b(\bfv,\bfw) &= \sum_{i=1}^2 b_i(\bfv,\bfw) 
\\
c(\bfv,\bfw) &=\sum_{i=1}^2 c_i(\bfv,\bfw) 
\\
l(\bfv) &= \sum_{i=0}^2 l_i(\bfv_i)
\end{align}
where 
\begin{equation}
I_a =
\begin{cases}
\{1, 2\} & \text{Case 1}
\\
\{0, 1, 2\} & \text{Case 2}
\end{cases}
\end{equation}
is the index set indicating the number of active forms in Case 1 (auxiliary hybrid space) 
and Case 2 (physical modeling hybrid space), 
\begin{equation}\label{eq:form-b}
b_i(\bfv,\bfw)
=
\gamma_0^{-1} h_i (S_i(\bfv),  DS_i(\bfw))_{\partial \Omega_{i,C}} 
\end{equation}
with $S_i$ defined in (\ref{eq:S}) and
\begin{equation}
DS_i(\bfw)= \sigma_n(\bfw_i) - \gamma_0 h_i^{-1} [w_n]_i 
\end{equation}
and 
\begin{equation}\label{eq:form-c}
c_i(\bfv,\bfw) = \gamma_0^{-1} h_i (\sigma_n(\bfv_i),\sigma_n(\bfw_i))_{\partial \Omega_{i,C}}
\end{equation}

\section{Error Estimates}
\label{se:error-estimates}
\subsection{Properties of the Forms}
\begin{itemize}

\item The forms $a_i$,  $i\in I_a$,  are coercive and continuous
\begin{align}\label{eq:aicoer}
\| \bfv \|_{V_i}^2 &\lesssim a_i(\bfv,\bfv) \qquad \bfv \in V_i
\\ \label{eq:aicont}
a_i(\bfv_i,\bfw_i) &\lesssim \| \bfv \|_{V_i} \|\bfw \|_{V_i}  \qquad \bfv,\bfw \in V_i
\end{align}
For $i\in \{1,2\}$,  $V_i = H^1(\Omega_i)$ and $V_0$ will depend on the choice 
of hybrid model. For a second order model $V_0 = H^1(\Omega_0)$ and for a fourth order 
model $V_0 = H^2(\Omega_0)$. We define the energy norm
\begin{align}\label{eq:a-energynorm}
\| \bfv \|_{a}^2 = a(\bfv,\bfv) = \sum_{i \in I_a} a_i(\bfv_i,\bfv_i)
\end{align}
%
%
\item The forms $b_i$ satisfies the monotonicity  and continuity 
\begin{gather}\label{eq:b-coer-i}
\gamma_0^{-1} \| S_i(\bfv) - S_i(\bfw) \|^2_{H^{-1/2}_h(\partial \Omega_{i,C})}
\leq b_i(\bfv,\bfv-\bfw)  -  b_i(\bfw,\bfv-\bfw)
\\ \label{eq:b-cont-i}
|b_i(\bfv,\bfr) - b_i(\bfw,\bfr)| 
\leq 
\gamma_0^{-1} \| S_i(\bfv) - S_i(\bfw) \|_{ H^{-1/2}_h( \partial \Omega_{i,C} ) }\| DS_i(\bfr) \|_{H^{-1/2}_h( \partial \Omega_{i,C})}
\end{gather}
with vectorized versions, 
\begin{gather}\label{eq:b-coer}
\sum_{i=1}^2 \gamma_0^{-1} \| S_i(\bfv) 
-S_i(\bfw) \|^2_{H^{-1/2}_h(\partial \Omega_{i,C})}
\leq b(\bfv,\bfv-\bfw)  -  b(\bfw,\bfv-\bfw)
\\ \label{eq:b-cont}
|b(\bfv,\bfr) - b(\bfw,\bfr)| \leq \sum_{i=1}^2 \gamma_0^{-1} \| S_i(\bfv) - S_i(\bfw) \|_{ H^{-1/2}_h(\partial \Omega_{i,C}) } 
\| DS_i(\bfr) \|_{H^{-1/2}_h(\partial \Omega_{i,C})}
\end{gather}
See the appendix for derivations of these inequalities.

\item The form $c$ is continuous
\begin{align}\label{eq:c-cont}
c(\bfv,\bfw) 
%
&\lesssim \| \bfv \|_c \| \bfw \|_c
\qquad  \bfv, \bfw \in W+ W_{h}
\end{align}
and we have the inverse bound
\begin{align} 
\| \bfv \|^2_c &= \sum_{i=1}^2 \gamma_0^{-1}  \| \sigma_n(\bfv_i) \|_{H_h^{-1/2}(\partial \Omega_{i,C})}^2
\\
&\qquad \qquad \label{eq:c-cont-b}
\lesssim \sum_{i=1}^2\gamma_0^{-1} \| \bfv_i \|^2_{V_i} 
\lesssim \gamma_0^{-1} \| \bfv \|^2_a 
\qquad \bfv \in W_h
\end{align}
where the last bound follows from (\ref{eq:aicoer}) for $i \in \{1,2\}$ and the fact that $\{1,2\} \subset I_a$,  in both Case 1 and Case 2.

\end{itemize}

\subsection{Estimates}

\begin{prop}The discrete problem (\ref{eq:fem}) admits a unique solution 
$\bfu_h \in W_h$ such that
\begin{align}
& \|\bfu_{h}\|^2_a 
+
\sum_{i=1}^2 \gamma_0^{-1} \|S_i(\bfu_h)  - S_i(\bfzero) \|_{H^{-1/2}_h (\partial \Omega_{i,C})}^2 
 \\
&\qquad \qquad \lesssim\sum_{i \in I_a} \|\bff_i \|^2_{V_i^*} 
+ \sum_{i=1}^2 \gamma_0^{-1} h \| S_i(\bfzero) \|^2_{\partial \Omega_{i,C}}
\end{align}
\end{prop}
\begin{proof}  We have 
\begin{align}
l(\bfv) = A(\bfv,\bfv) &=  \|\bfv\|_a^2 + b(\bfv,\bfv) - \|\bfv\|_c^2  
\end{align}
where 
\begin{align}
b(\bfv,\bfv) &= b(\bfv,\bfv-\bfzero) - b(\bfzero,\bfv -\bfzero) + b(\bfzero,\bfv)
\end{align}
and thus
\begin{align}
\|\bfv\|_a^2  - \|\bfv\|_c^2  + b(\bfv,\bfv-\bfzero) - b(\bfzero,\bfv -\bfzero) 
= l(\bfv) - b(\bfzero,\bfv)
\leq |l(\bfv)| + |b(\bfzero,\bfv)|
\end{align}
Here we have using (\ref{eq:c-cont-b}), for $\gamma_0$ large enough, 
\begin{align}
 \|\bfv\|_a^2 - \|\bfv\|_c^2  
 \geq  \|\bfv\|_a^2 - C_1  \gamma_0^{-1} \|\bfv\|_a^2
 \geq (1 - C_1 \gamma_0^{-1} ) \|\bfv\|_a^2 
 \geq C_2  \|\bfv\|_a^2 
\end{align}
with $C_2>0$ and by (\ref{eq:b-coer}), 
\begin{align}
b(\bfv,\bfv-\bfzero) - b(\bfzero,\bfv -\bfzero) 
\geq \sum_{i=1}^2 \gamma_0^{-1} \|S_i(\bfv_{i})  - S_i(\bfzero) \|_{H^{-1/2}_h( \partial \Omega_{i,C}) }^2 
\end{align}
We thus have
\begin{align}\label{eq:prf-stab-b}
C_2  \|\bfv\|_a^2 +  \sum_{i=1}^2 \gamma_0^{-1} \|S_i(\bfv_{i})  - S_i(\bfzero) \|_{H^{-1/2}_h( \partial \Omega_{i,C}) }^2  
\leq  |l(\bfv)| + |b(\bfzero,\bfv)|
\end{align}
Next, estimating the right hand side,  we have 
\begin{align}
|l(\bfv)| \lesssim \|\bff \|_{W^*} \| \bfv \|_W \lesssim \|\bff \|_{W^*} \| \bfv \|_a
\end{align}
and
\begin{align}
|b(\bfzero,\bfv)|
&\leq
 \sum_{i=1}^2 \gamma_0^{-1} \| S_i (\bfzero )\|_{H^{-1/2}_h(\partial \Omega_{i,C})} \| DS_i (\bfv) \|^2_{H^{-1/2}_h(\partial \Omega_{i,C})}
 \\
 &=
 \sum_{i=1}^2 \gamma_0^{-1} \| S_i (\bfzero )\|_{H^{-1/2}_h(\partial \Omega_{i,C})} \| S_i (\bfv) - S_i(\bfzero) \|^2_{H^{-1/2}_h( \partial \Omega_{i,C} )}
\\
& \label{eq:prf-stab-c}
\leq \sum_{i=1}^2 \frac12 \gamma_0^{-1} \| S_i (\bfzero )\|^2_{H^{-1/2}_h ( \partial \Omega_{i,C} )} 
 +
\frac12 \gamma_0^{-1} \| S_i (\bfv) - S_i(\bfzero) \|^2_{H^{-1/2}_h ( \partial \Omega_{i,C})}
\end{align}
where we used the identity $S_i(\bfv) = S_i(\bfzero) + DS_i(\bfv)$,  which holds for affine maps. Using kick-back for the second term in (\ref{eq:prf-stab-c}) 
these estimates prove the desired stability estimate.  Using the Brouwer fixed point theorem, we can prove that there exists a solution; see \cite{BHLS17} for details. Uniqueness follows from the stability estimate.
\end{proof}

\begin{thm}The solution $\bfu_h \in W_h$ to (\ref{eq:fem}) satisfies the following best approximation estimate
\begin{align}
& \|\bfu - \bfu_h\|_a^2 
+ \sum_{i=1}^2 \gamma_0^{-1}  \|S_i(\bfu) -  S_i(\bfu_h) \|_{H^{-1/2}_h( \partial \Omega_{i,C})}^2
\\
&\qquad
\lesssim \inf_{\bfv_h \in W_h} \Big(
( 1 + \gamma_0^{-1}) \|\bfu - \bfv_h\|_a^2 
\\
&\qquad \qquad \qquad
+ \sum_{i=1}^2 \gamma_0^{-1} \|\sigma_n(\bfu_i -\bfv_{h,i}) \|_{H^{-1/2}_h(\partial \Omega_{i,C})}^2 
+ \gamma_0 \|[\bfu-\bfv_h]_i\|_{H^{1/2}_h(\partial \Omega_{i,C})}^2
\Big)
\end{align}
\end{thm}

\begin{proof} We first note that we have the Galerkin orthogonality
\begin{align}\label{eq:galort}
a(\bfu-\bfu_h,\bfv) + b(\bfu,\bfv) - b(\bfu_h,\bfv) - c(\bfu-\bfu_h,\bfv) = 0, \qquad \forall \bfv \in W_h
\end{align}
Next,  we split the error,  $\bfe = \bfu - \bfu_h$,  in an approximation error $\bftheta$ and a discrete error $\bfe_h \in W_h$,
\begin{align}
\bfe := \bfu - \bfu_h = (\bfu - \bfv_h) + (\bfv_h - \bfu_h) = \bftheta + \bfe_h, \qquad \bfv_h \in W_h
\end{align}
Using the coercivity (\ref{eq:a-energynorm}) of $a$, monotonicity (\ref{eq:b-coer}) of $b$, followed by Galerkin orthogonality (\ref{eq:galort}), we get
\begin{align}
&\| \bfe \|_a^2 + \sum_{i=1}^2 \gamma_0^{-1}  \| S_i(\bfu) - S_i(\bfu_h) \|^2_{H^{-1/2}_h(\partial \Omega_{i,C})}
\\
&\leq
a(\bfe,\bfe) + b(\bfu,\bfe) - b(\bfu_h,\bfe) 
\\
&=a(\bfe,\bftheta) + b(\bfu,\bftheta) - b(\bfu_h,\bftheta)
 + \underbrace{a(\bfe,\bfe_h) + b(\bfu, \bfe_h) - b(\bfu_h, \bfe_h)}_{=c(\bfe,\bfe_h) }
\\
&=a(\bfe,\bftheta) + b(\bfu,\bftheta) - b(\bfu_h,\bftheta)  + c(\bfe,\bfe_h) 
\\ \label{eq:errest-a}
&\leq
\| \bfe \|_a \| \bftheta \|_a + \sum_{i=1}^2 \gamma_0^{-1}  \| S_i(\bfu) -  S_i(\bfu_h) \|_{H^{-1/2}_h(\partial \Omega_{i,C})} \|DS_i( \bftheta )\|_{H^{-1/2}_h(\partial \Omega_{i,C})} 
\\
&
\qquad +   \frac12  \|\bftheta\|_c^2 +  \frac32 C \gamma_0^{-1} \|\bftheta \|^2_a +  \frac32 C \gamma_0^{-1} \|\bfe \|^2_a
\end{align}
In (\ref{eq:errest-a}) we used the continuity of $a$ and $b$, and the estimate 
\begin{align}
 c(\bfe,\bfe_h) &=  c(\bftheta,\bfe_h)  + c(\bfe_h,\bfe_h) 
 \\
 &\leq  \|\bftheta\|_c \|\bfe_h\|_c  +  \|\bfe_h \|_c^2
 \\
 &\leq \frac12  \|\bftheta\|_c^2 + \frac32 \|\bfe_h \|_c^2
 \\
  &\leq \frac12 \|\bftheta\|_c^2 + \frac32 C \gamma_0^{-1} \|\bfe_h \|^2_a
   \\
  &\leq  \frac12  \|\bftheta\|_c^2 +  \frac32 C \gamma_0^{-1} \|\bftheta \|^2_a +  \frac32 C \gamma_0^{-1} \|\bfe \|^2_a
\end{align}
where we used the inverse estimate (\ref{eq:c-cont-b}). Splitting the first and second terms in (\ref{eq:errest-a}) using $ab \leq a^2/2 + b^2/2$ 
and taking $\gamma_0$ large enough, we may use a kick-back argument to get 
\begin{align}
&\| \bfe \|_a^2 + \sum_{i=1}^2 \gamma_0^{-1} \| S_i(\bfu) - S_i(\bfu_h) \|^2_{H^{-1/2}_h(\partial \Omega_{i,C})}
\\
&\qquad \lesssim 
 \| \bftheta \|^2_a + \sum_{i=1}^2 \gamma_0^{-1}   \|DS_i( \bftheta )\|^2_{H^{-1/2}_h(\partial \Omega_{i,C})} + \gamma_0^{-1} \| \bftheta \|_{a}^2 + \| \bftheta\|_c^2
\\
&\qquad \lesssim 
(1 + \gamma_0^{-1} ) \| \bftheta \|^2_a
+ \sum_{i=1}^2 \gamma_0^{-1}  \|\sigma_n (\bftheta_i) \|^2_{H^{-1/2}_h(\partial \Omega_{i,C})} 
 +  \gamma_0  \| [\theta_n]_i \|^2_{H^{1/2}_h(\partial \Omega_{i,C})} 
\end{align}
where we used the triangle inequality to conclude that 
\begin{align}
\gamma_0^{-1} \|DS_i(\bfv)\|^2_{H^{-1/2}_h(\partial \Omega_{i,C})} 
&=
\gamma_0^{-1} h \|\sigma_n(\bfv_i) - \gamma_0 h^{-1} [v_n]_i\|^2_{\partial \Omega_{i,C}} 
\\
&\leq
\gamma_0^{-1} h \|\sigma_n(\bfv_i) \|^2_{\partial \Omega_{i,C}}  + \gamma_0 h^{-1} \| [v_n]_i\|^2_{\partial \Omega_{i,C}} 
\\
&=
\gamma_0^{-1}  \|\sigma_n(\bfv_i) \|^2_{H^{-1/2}_h(\partial \Omega_{i,C})}  + \gamma_0 \| [v_n]_i\|^2_{H^{1/2}_h(\partial \Omega_{i,C})} 
\end{align}
Thus, the proof is complete.
\end{proof}
\begin{rem} Letting $\mcT_h(\partial \Omega_{i,C})$ denote the elements in $\mcT_{h,i}$ with 
a face that intersects $\partial \Omega_{i,C}$ and using the element-wise trace inequality 
$\| w \|^2_{\partial T} \lesssim h^{-1} \| w \|^2_T + h \| \nabla w \|^2_T$ for $w \in H^1(T)$, we obtain
\begin{align}
&\gamma_0^{-1} h_i \|\sigma_n(\bfv_i) \|^2_{\partial \Omega_{i,C}}  + \gamma_0 h_i^{-1} \| [v_n]_i\|^2_{\partial \Omega_{i,C}} 
\\
&\lesssim
\gamma_0^{-1} h_i \|\sigma_n(\bfv_i)\|^2_{\partial \Omega_{i,C}}
+
\gamma_0  h_i^{-1}  \| v_{i,n} \|^2_{\partial \Omega_{i,C}} 
+ 
\gamma_0 h_i^{-1}  \| v_{0,n} \|^2_{\Omega_0} 
\\
&\lesssim 
 \gamma_0^{-1} (
\| \nabla \bfv_i \|^2_{\mcT_h(\partial \Omega_{i,C})}
+ 
h_i^2 \| \nabla^2 \bfv_i \|^2_{\mcT_h(\partial \Omega_{i,C})}
)
\\ 
&\qquad 
+
\gamma_0 (
h_i^{-2} \| \bfv_i \|^2_{\mcT_h(\partial \Omega_{i,C})} 
+ 
\| \nabla \bfv_i \|^2_{\mcT_h(\partial \Omega_{i,C})}
)
+
\gamma_0 h_i^{-1} \| v_{0,n} \|^2_{\Omega_0}
\end{align}
Setting $\bfv = \bftheta = \bfu - \bfv_h$,  and taking $\bfv_h$ to an interpolant $\pi_{h} \bfu$ 
of the exact solution $\bfu$ that satisfies optimal order interpolation error 
bounds (\ref{eq:interpol}), we conclude that our best approximation result, indeed 
leads to the optimal order energy error estimate,
\begin{align}
\|\bfu - \bfu_h\|_a^2 
&+ \sum_{i=1}^2 \gamma_0^{-1}  \|S_i(\bfu) -  S_i(\bfu_h) \|_{H_h^{-1/2}(\partial \Omega_{i,C})}^2
\\
&\qquad
\lesssim 
\begin{cases}
\sum_{i=1}^2 h_0^{2 p_0 +1} (h_0/h_i) + \sum_{i=1}^2 h_i^{2 p_i} & \text{Case 1}
\\
h_0^{2(p_0+1-s_0)} +  \sum_{i=1}^2 h_0^{2 p_0 +1} (h_0/h_i)+ \sum_{i=1}^2 h^{2 p_i} & \text{Case 2}
\end{cases}
\end{align}
Here,  in Case 2,  $a_0$ is a form defined on $H^{s_0}(\Omega_0)$ and we assume that $V_{h,0} \subset H^{s_0}(\Omega_0)$.
In Case 1 when  $a_0$ is not present the estimate simplifies. For both cases we can take piecewise constants in 
$V_{0,h}$ and still get convergence,  see the numerical example Section \ref{sec:constants}.  Note that we can take $h_0$ arbitrarily small 
compared to $h_i$,  but not the other way around.
\end{rem}

\section{Numerical Examples}
\label{sec:numerics}

One domain is fixed using Dirichlet boundary conditions in the numerical examples presented. To remove the rigid body motions of the other, we use scalar.
Lagrange multipliers ensure that the mean horizontal displacements are zero.

\subsection{A 2D Hertz Problem}
\label{sec:constants}
We consider a half-sphere in contact with a rectangular block. The geometry and the meshes used are given in Fig. \ref{fig:meshes}. The block is fixed vertically at its bottom boundary and horizontally at $(x,y)=(0,-1)$.
The half-sphere is loaded by a line load $\bff=(0,-50)$ at $y=1$ and has elastic moduli $E=2000$ and $\nu=0.3$. It is meshed using $P^1$ triangles. The block has elastic moduli $E=7000$, $\nu=0.3$ and is discretized using $Q_1$ elements. 
We assume plane strain so that $\lambda_i=\nu_iE_i/((1+\nu_i)(1-2\nu_i))$ and $\mu_i=E_i/(2(1+\nu_i))$,
and we set $\gamma_i=10E_i$.
For the approximation of $\bfn\cdot\bfu_0$ we use,
for illustrative purposes, the suboptimal choice of equidistributed piecewise constants on $x\in (-0.4,0.4)$. This approximation is tied to the upper boundary of the block by an equality constraint. In Figs. \ref{fig:50close}--\ref{fig:1000close} we show close-ups of part of the contact zone in the deformed state for different numbers of constants. Finally, in Figs. \ref{fig:50press}--\ref{fig:1000press} we show a comparison with the corresponding Hertz solution for two cylinders in contact, the upper with radius $r=1$ and the lower with $r=\infty$ (cf., e.g., \cite{Popov10}). This is not identical to the setup in the problem we solve but indicates the accuracy of the solution.

\begin{figure}
  \centering
\includegraphics[width=5in]{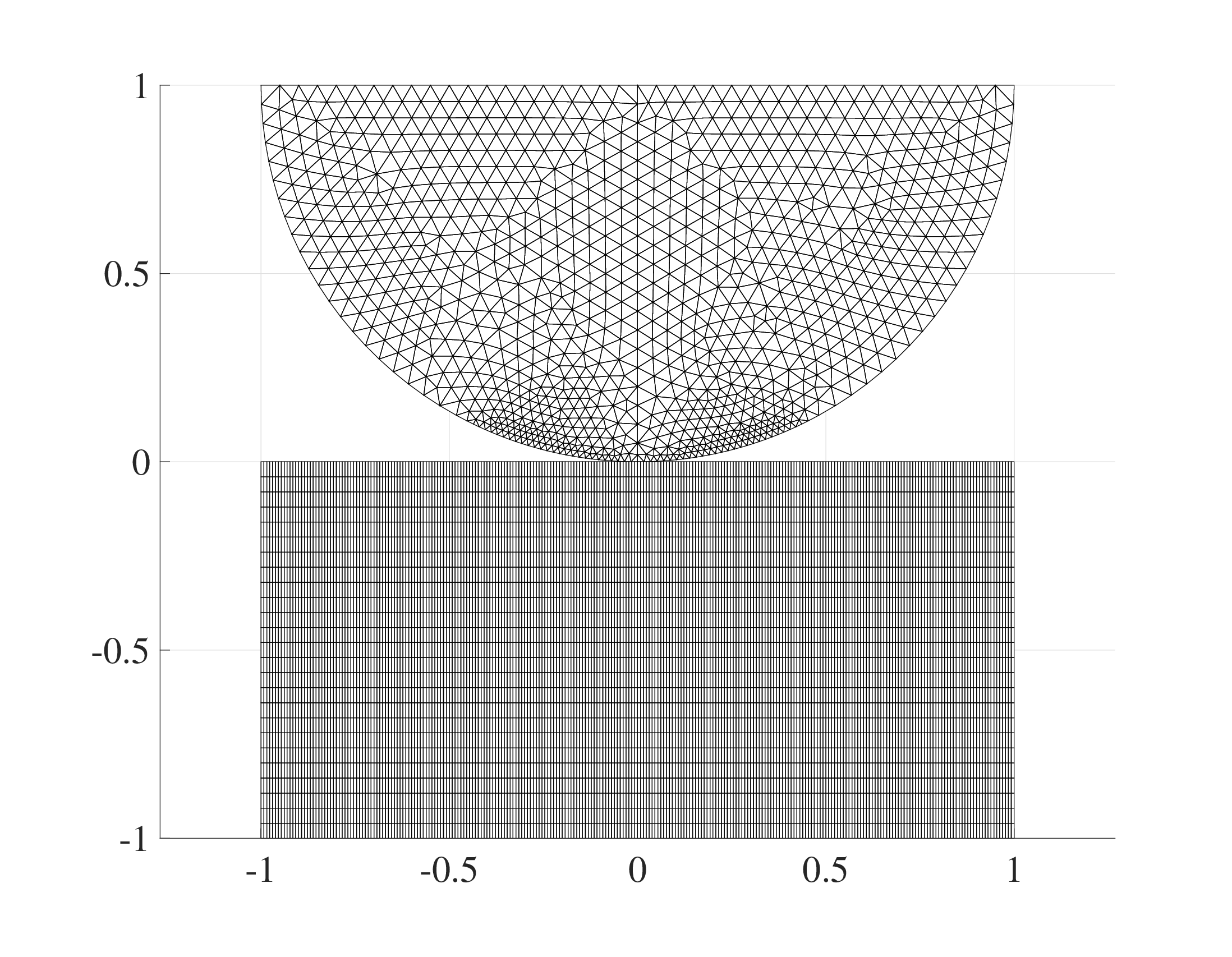}
    \caption{Geometry and meshes used for the Hertz problem.}
  \label{fig:meshes}
\end{figure}
\begin{figure}
  \centering
\includegraphics[width=5in]{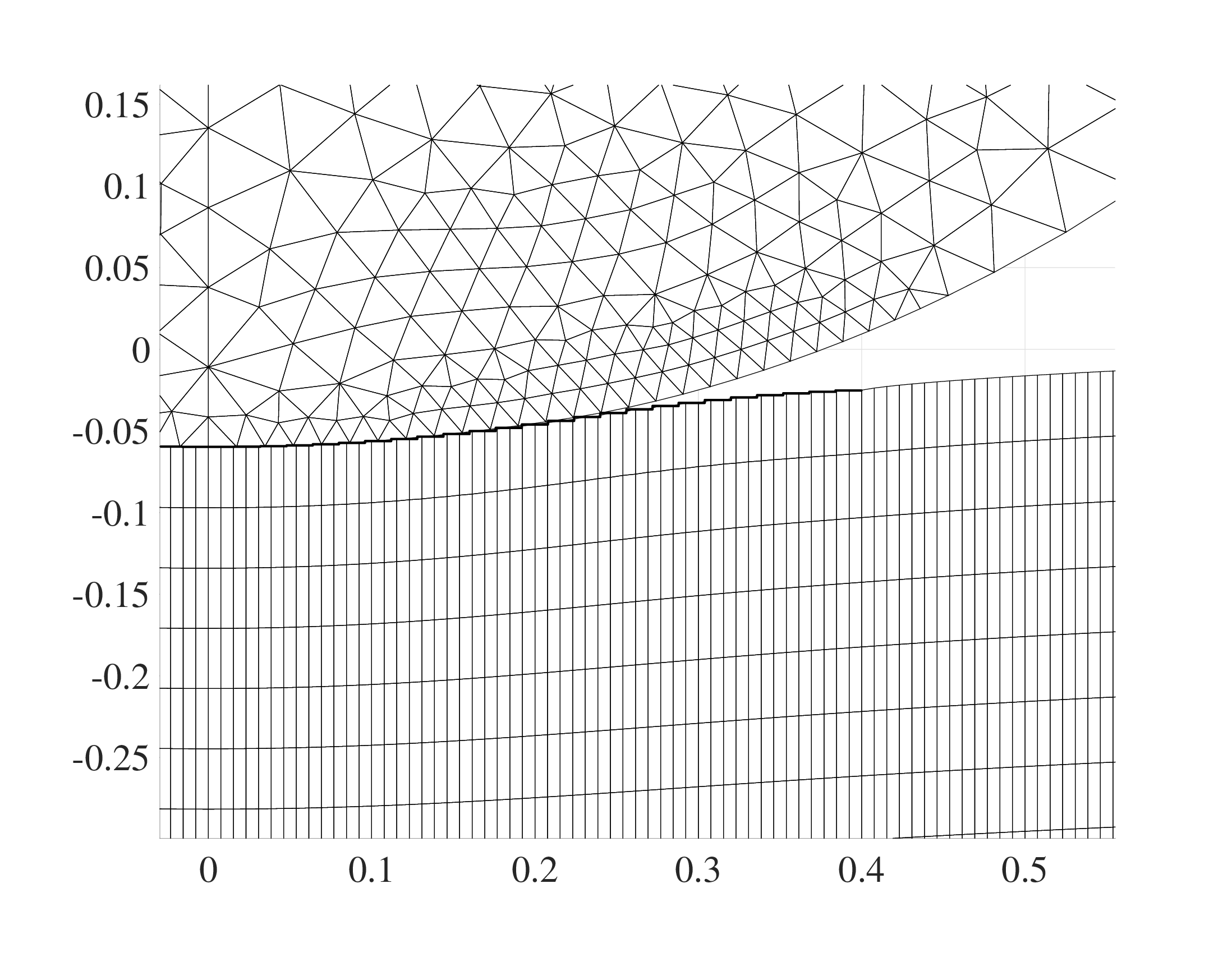}
    \caption{Close up of the solution with 50 constants.}
  \label{fig:50close}
\end{figure}
\begin{figure}
  \centering
\includegraphics[width=5in]{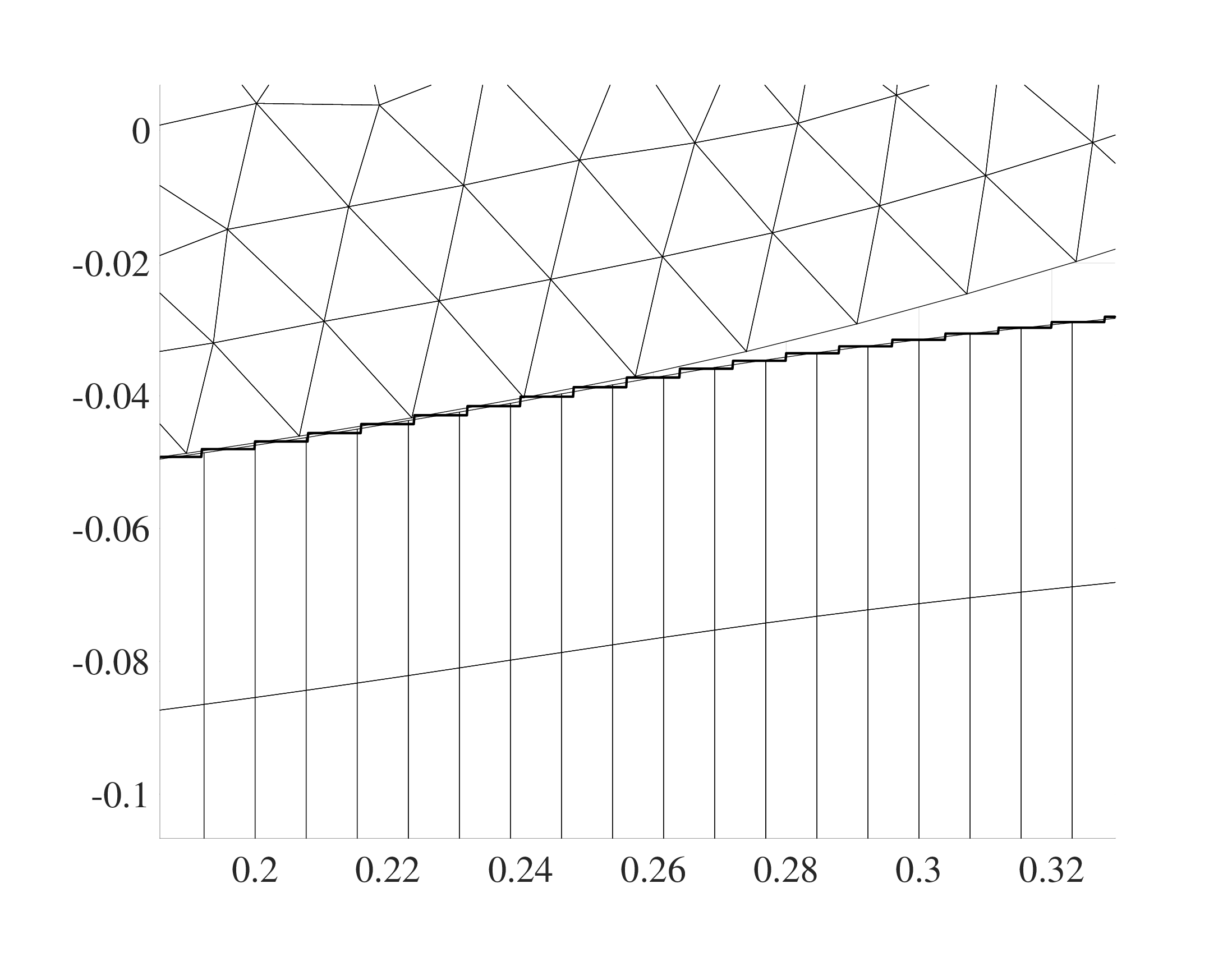}
    \caption{Cose up of the solution with 100 constants.}
  \label{fig:100close}
\end{figure}
\begin{figure}
  \centering
\includegraphics[width=5in]{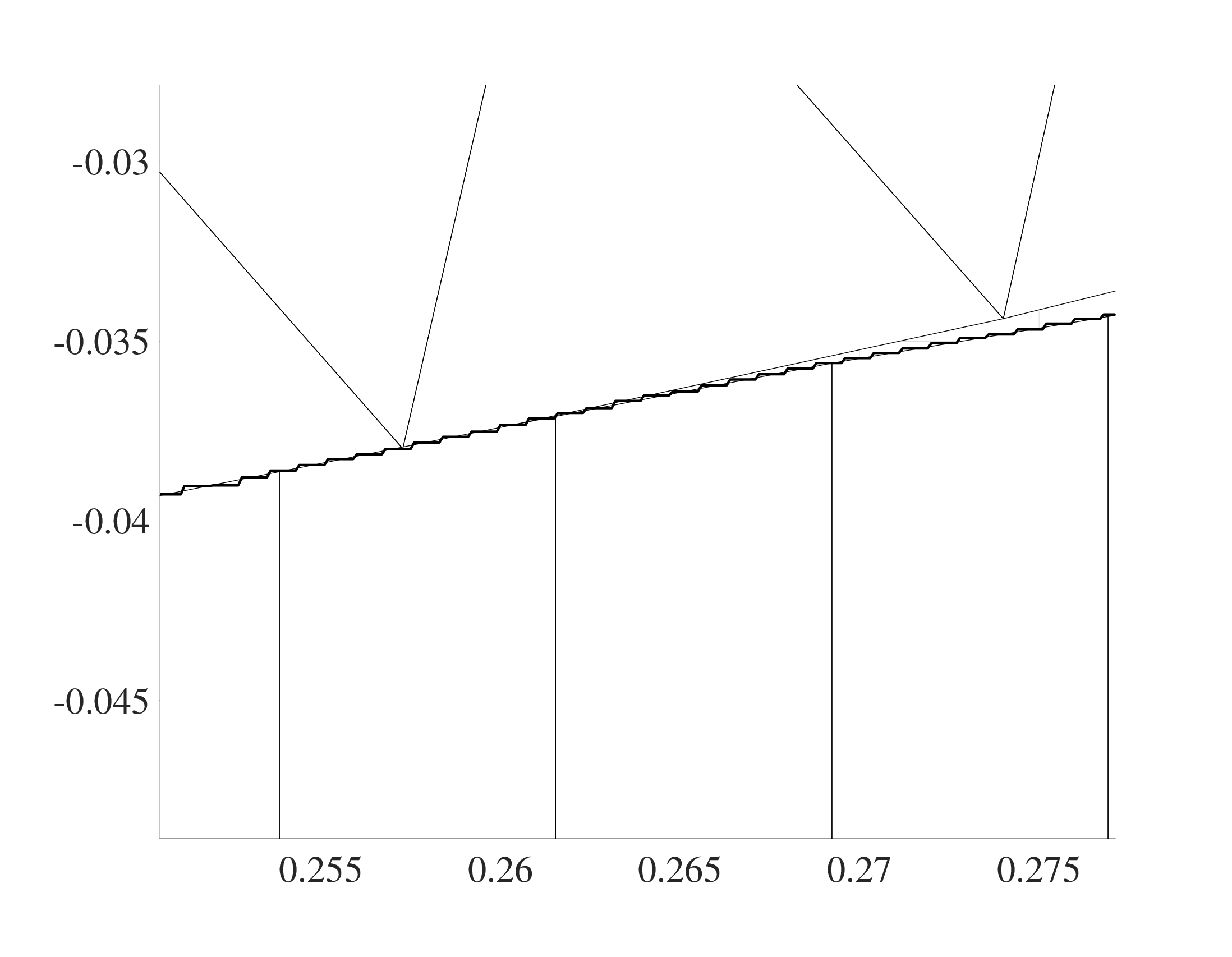}
    \caption{Close up of the solution with 1000 constants.}
  \label{fig:1000close}
\end{figure}
\begin{figure}
  \centering
\includegraphics[width=5in]{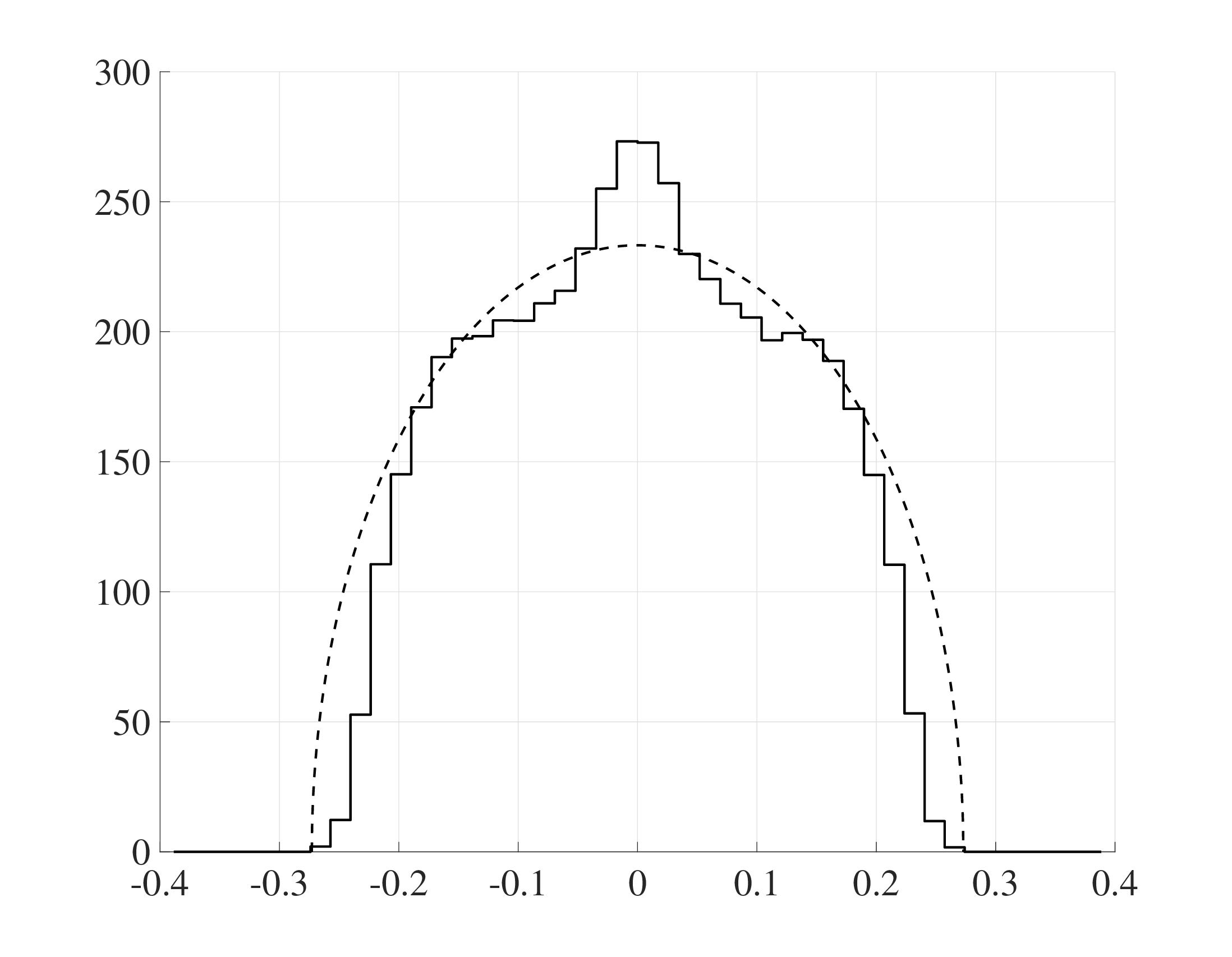}
    \caption{Computed and analytical contact pressures with 50 constants.}
  \label{fig:50press}
\end{figure}
\begin{figure}
  \centering
\includegraphics[width=5in]{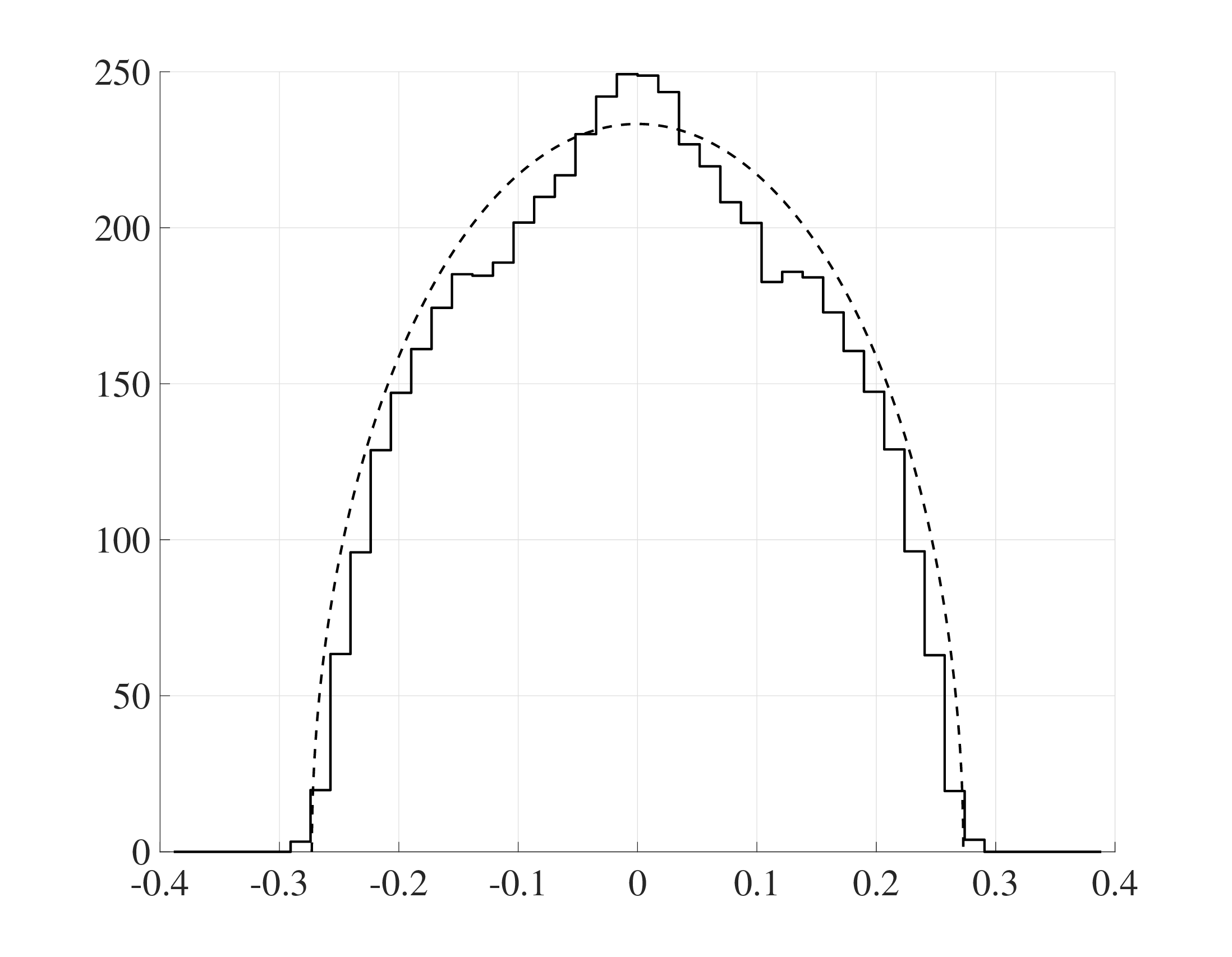}
    \caption{Computed and analytical contact pressures with 100 constants.}
  \label{fig:100press}
\end{figure}
\begin{figure}
  \centering
\includegraphics[width=5in]{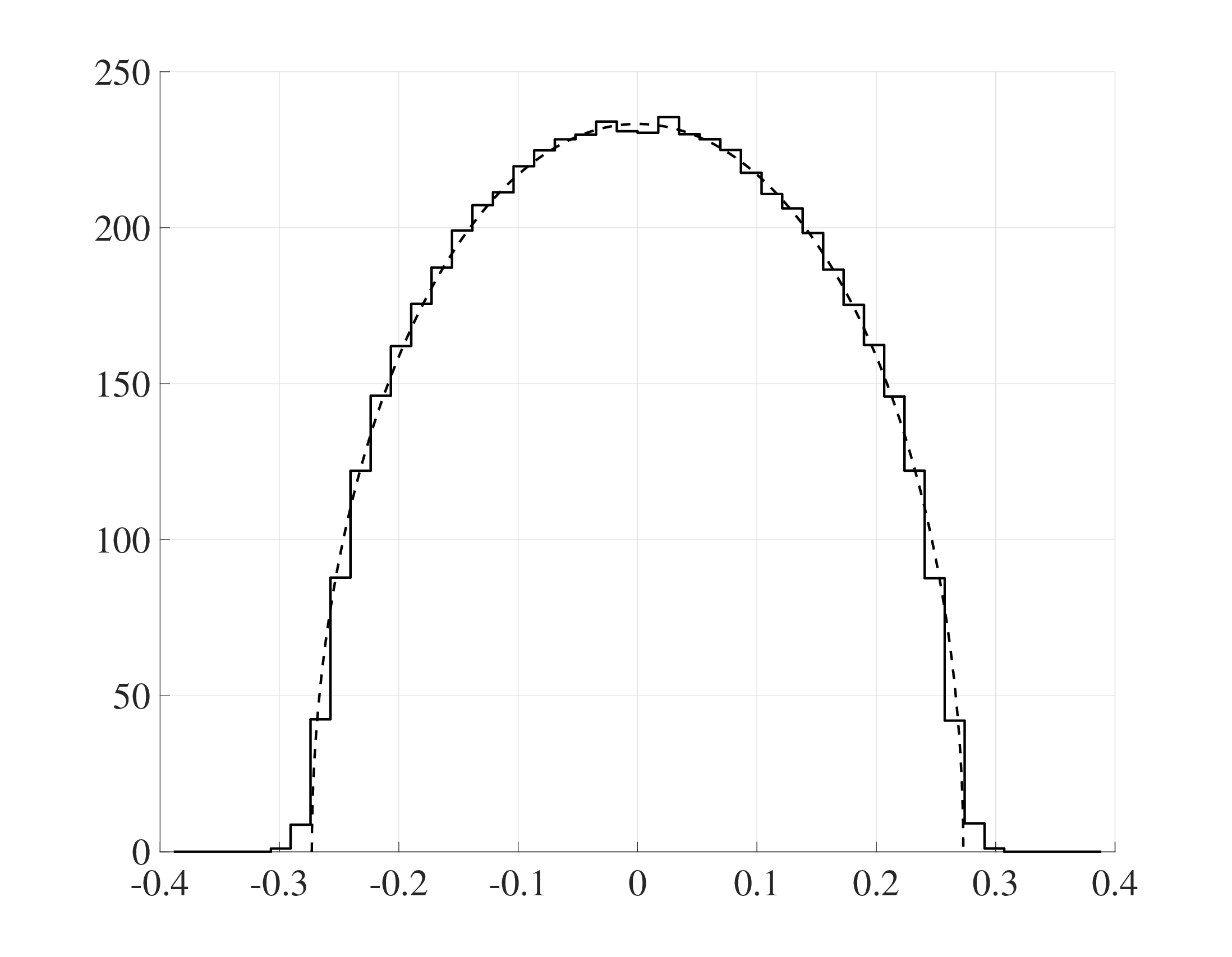}
    \caption{Computed and analytical contact pressures with 1000 constants.}
  \label{fig:1000press}
\end{figure}

\subsection{Contact Problems with Model Coupling}\label{sec:model}

\subsubsection{The Plate Model}

In this example, we use a plate model to calculate the hybrid variable. We consider the Kirchhoff plate model, posed on a rectangular domain 
$\Omega_0$, in the $(x,y)-$plane, where we seek an out--of--plane 
(scalar) displacement $u_0$, with $\bfu_0 = (0,0,u_0)$,  to which we associate the strain (curvature) tensor
\begin{equation}
\bfeps_p(\nabla u_0) := \frac12\left(\nabla\otimes (\nabla u_0) + (\nabla u_0 ) \otimes \nabla \right) 
= \nabla \otimes \nabla u_0 = \nabla^2 u_0   
\end{equation}
and the plate stress (moment) tensor
\begin{align}\label{eq:plate-stress-tensor}
\bfsig_p (\nabla u_0) := &D \left(\bfeps(\nabla u_0) + \nu (1- {\nu_0 })^{-1}
\div\nabla u _0\, \bfI \right)
\\
= &D \left( \nabla^2 u_0 + \nu (1-\nu_0)^{-1} \Delta u_0 \bfI \right)
\end{align}
where 
\begin{equation}\label{eq:mcCP}
D :=  \frac{E_0 t^3}{12(1+\nu_0)} 
\end{equation}
with $E_0$ the Young's modulus, $\nu_0$ the Poisson's ratio, and $t$ the 
plate thickness. 
The equilibrium equations of a free Kirch\-hoff plate
take the form
\begin{align}
\div \Div \bfsig_p ( \nabla u_0 )= 0 &  \qquad \text{in $\Omega_0$}
\end{align}
where $\Div$ and $\text{div}$ denote the divergence of a tensor and a vector 
field, respectively.
Multiplying by $v$ and integrating by parts, the form $a_0(\cdot,\cdot)$ is found as
\begin{equation}
a_0(u_0,v)  := (\bfsig_p(\nabla u_0), 
\bfeps_p(\nabla v))_{\Omega_0}
\end{equation}

We consider a problem consisting of $\Omega_1$ being a  stiff ball of radius $r=1$, $\Omega_2$ a block of dimensions $(-2.5,2.5)\times(-2.5,2.5)\times (-1,0)$, and $\Omega_0$ a plate resting on the block. Fig. \ref{fig:platemehses} shows the configuration and meshes used.
The ball has constitutive parameters $E_1=20000$, $\nu_1=0.33$, the block has $E_2=25$, $\nu=0.33$, and the plate has $E_0=1000$, $t=0.1$, and $\nu=0.5$. The ball is being pushed downwards by a volume force $\bff_1 = (0,0,-10)$.
We take $\gamma_i=100E_i$. The block and the ball are discretized using $P^1$ tetrahedral elements, and the plate is discretized using Bogner-Fox-Schmit elements \cite{BognerFoxSchmit65}. Dirichlet boundary conditions $\bfu_2={\bf 0}$
are applied to the block at the bottom and at the sides.

In Fig. \ref{fig:confdef}, we show the deformation of the whole configuration; in Fig. \ref{fig:platedef}, we show the deformation of the plate and the block.

\begin{figure}
  \centering
\includegraphics[width=5in]{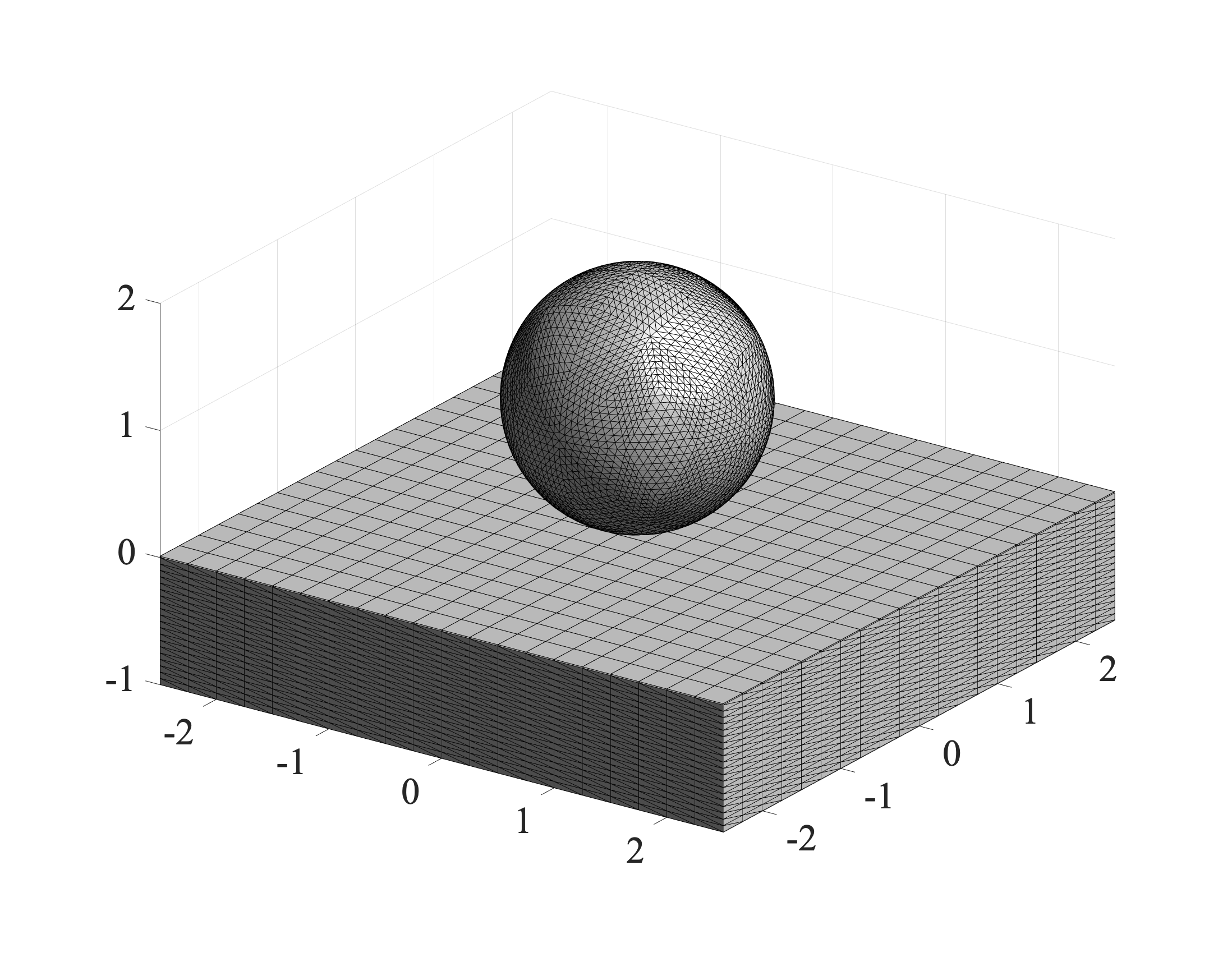}
    \caption{Geometry and meshes used for the plate contact problem.}
\label{fig:platemehses}
 \end{figure}
\begin{figure}
  \centering
\includegraphics[width=5in]{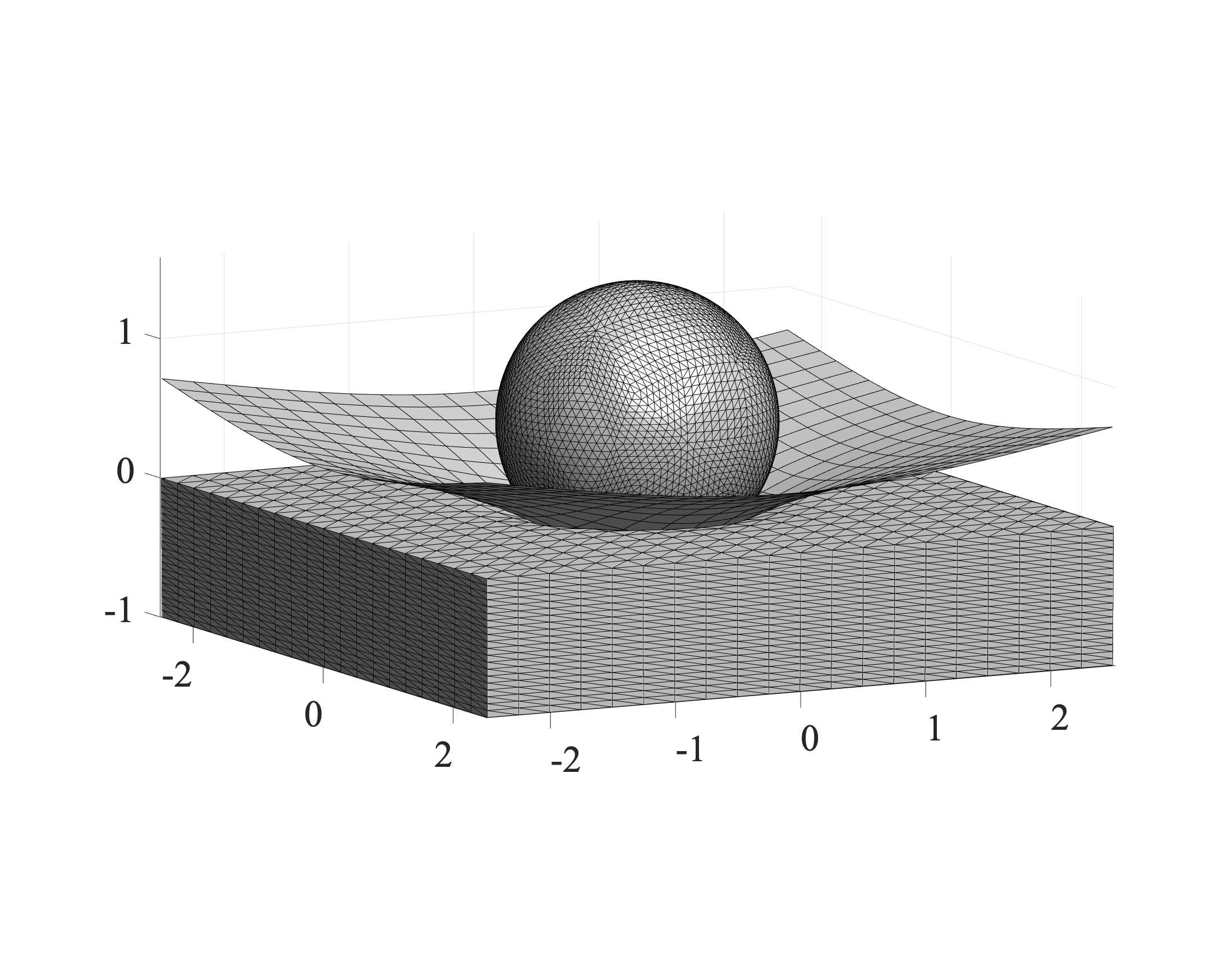}
    \caption{Deformation of the configuration.}
 \label{fig:confdef}
\end{figure}
\begin{figure}
  \centering
\includegraphics[width=3in]{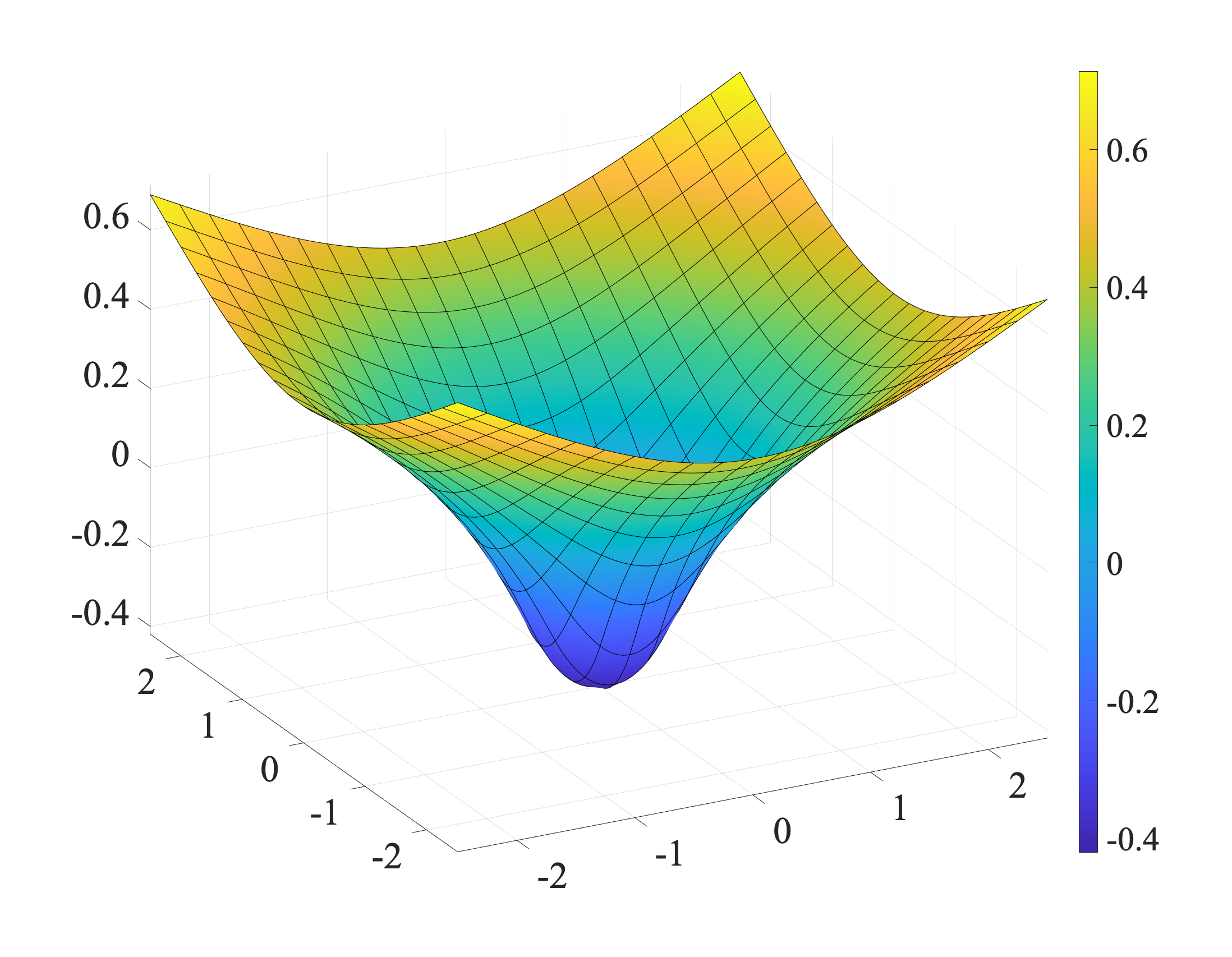}\includegraphics[width=3in]{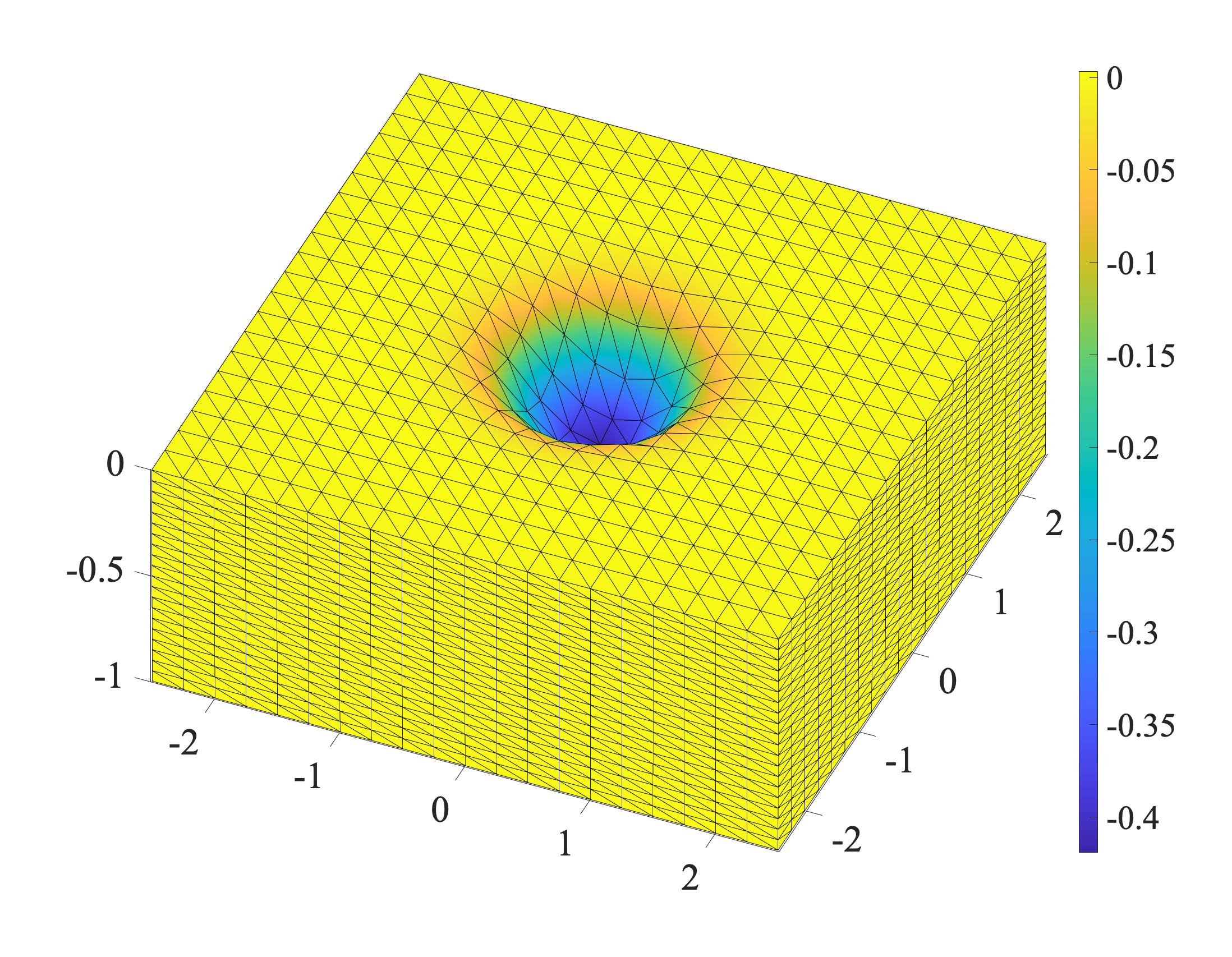}
    \caption{Deformation of the plate and the block.}
 \label{fig:platedef}
\end{figure}

\subsubsection{The Membrane Model}
In what follows, $\Gamma :=\Omega_0$ denotes a closed and oriented surface,
for simplicity without boundary, which is embedded in ${{\IR}}^{3}$ and equipped with exterior
normal $\bfn_\Gamma$. The membrane is assumed to occupy the domain $\Omega_t = \Gamma\times (-t/2,t/2)$ with $t$ the thickness of the membrane, assumed constant for simplicity.
We let $\rho$ denote the signed distance function
fulfilling $\nabla\rho\vert_{\Gamma} =\bfn_\Gamma$.

For a given function $u: \Gamma \to \IR$, we assume that there exists an extension $\bar{u}$, in some neighborhood of $\Gamma$, such that $\bar{u}\vert_\Gamma = u$.
 The the tangent gradient $\nabla_\Gamma$ on $\Gamma$ can be defined
by
\begin{equation}
\nabla_\Gamma u = \Ps \nabla \overline{u}
\label{eq:tangent-gradient}
\end{equation}
with $\nabla$ the ${{\IR}}^{3}$ gradient and $\Ps = \Ps(\bfx)$ the
orthogonal projection of $\IR^{3}$ onto the tangent plane of $\Gamma$ at $\bfx \in \Gamma$
given by
\begin{equation}
  \Ps = \bfI - \bfn_{\Gamma} \otimes \bfn_{\Gamma}
\end{equation}
where $\bfI$ is the identity matrix. 
The tangent gradient defined by \eqref{eq:tangent-gradient} is easily shown to be independent of the extension $\overline{u}$.
In the following, we shall not distinguish between functions on $\Gamma$ and their extensions when defining differential operators.

The surface gradient has three
components, which we shall denote by
\begin{equation}
\nabla_\Gamma u=: \left(
\frac{\partial u}{\partial x^\Gamma} ,
\frac{\partial u}{\partial y^\Gamma} ,
\frac{\partial u}{\partial z^\Gamma}\right) 
\end{equation}
For a vector-valued function $\bfv(\bfx)$, we define the
the tangential Jacobian matrix as the transpose of the outer product of $\nabla_\Gamma$ and $\bfv$,
\begin{equation}
\left(\nabla_\Gamma\otimes\bfv\right)^{\text{T}} :=\left[\begin{array}{>{\displaystyle}c>{\displaystyle}c>{\displaystyle}c}
\frac{\partial v_1}{\partial x^\Gamma} &\frac{\partial v_1}{\partial y^\Gamma} & \frac{\partial v_1}{\partial z^\Gamma} \\[3mm]
\frac{\partial v_2}{\partial x^\Gamma} &\frac{\partial v_2}{\partial y^\Gamma} & \frac{\partial v_2}{\partial z^\Gamma} \\[3mm]
\frac{\partial v_3}{\partial x^\Gamma} &\frac{\partial v_3}{\partial y^\Gamma} & \frac{\partial v_3}{\partial z^\Gamma}
\end{array}\right] ,
\end{equation}
the surface divergence $\nabla_{\Gamma}\cdot\bfv := \text{tr}\nabla_\Gamma\otimes\bfv$, and the in-plane strain tensor
\begin{equation}
\bfeps_{\Gamma}(\bfu) := \Ps\bfeps(\bfu)\Ps ,\quad\text{where}\quad \bfeps(\bfu) := \frac12\left(\nabla\otimes \bfu + (\nabla \otimes\bfu)^{\rm T}\right) 
\end{equation}
is the 3D strain tensor. The corresponding stress tensor is given by
\begin{equation}
\bfsig_\Gamma  =  2\mu_\Gamma \bfeps_\Gamma + {\lambda_\Gamma} \text{tr}\bfeps_\Gamma\, \Ps
\end{equation}
where, with Young's modulus $E_\Gamma$ and Poisson's ratio $\nu_\Gamma$, 
\begin{equation}
\mu_\Gamma = \frac{E_\Gamma t}{2(1+\nu_\Gamma)},\quad {\lambda_\Gamma}= \frac{E_\Gamma\nu_\Gamma t}{1-\nu_\Gamma^2}
\end{equation}
are the Lam\'e parameters in plane stress (multiplied by the thickness). 

The equilibrium equations on an unloaded membrane are given by (cf. \cite{HaLa14})
\begin{equation}
  \label{eq:LB}
-\nabla_\Gamma\cdot \bfsig_{\Gamma}(\bfu)   = {\bf 0}\quad \text{on $\Gamma$,}
\end{equation}
By multiplying the equilibrium equation by $\bfv$ and integrating by parts, we find the form $a_0(\cdot,\cdot)$ as 
\begin{equation}
a_0(\bfu,\bfv)=(2\mu_\Gamma\bfeps_{\Gamma}(\bfu),\bfeps_{\Gamma}(\bfv))_{\Gamma}+(\lambda_\Gamma\nabla_{\Gamma}\cdot\bfu,\nabla_{\Gamma}\cdot\bfv)_{\Gamma}
\end{equation}

We consider a problem with $\Omega_1$ a ball and $\Omega_2$ a block of the same dimensions as in the previous Section. The ball and block have fixed constitutive parameters $E_1=1000$, $\nu_1=0.33$ and $E_2=100$, $\nu_2=0.3$.
The ball is covered by a membrane with $t=0.1$ and $\nu_\Gamma=0.5$, whose Young's modulus we vary to show its stiffening effect. The discretization of the ball and the block are as in the previous section, and we let the surface mesh of the ball
serve as a $P^1$ mesh for the membrane. We use $\gamma_i=100E_i$. The ball is again loaded with a volume load $\bff_2=(0,0,-10)$.

In Fig. \ref{fig:zero}, we show the deformations when the membrane has zero stiffness (acts as a standard hybrid variable), in Fig. \ref{fig:2000}, we show the deformations when $E_\Gamma=2000$, and, finally, in Fig. \ref{fig:20000} when $E_\Gamma=20000$. The stiffening effect is clearly visible.

\begin{figure}
  \centering
\includegraphics[width=5in]{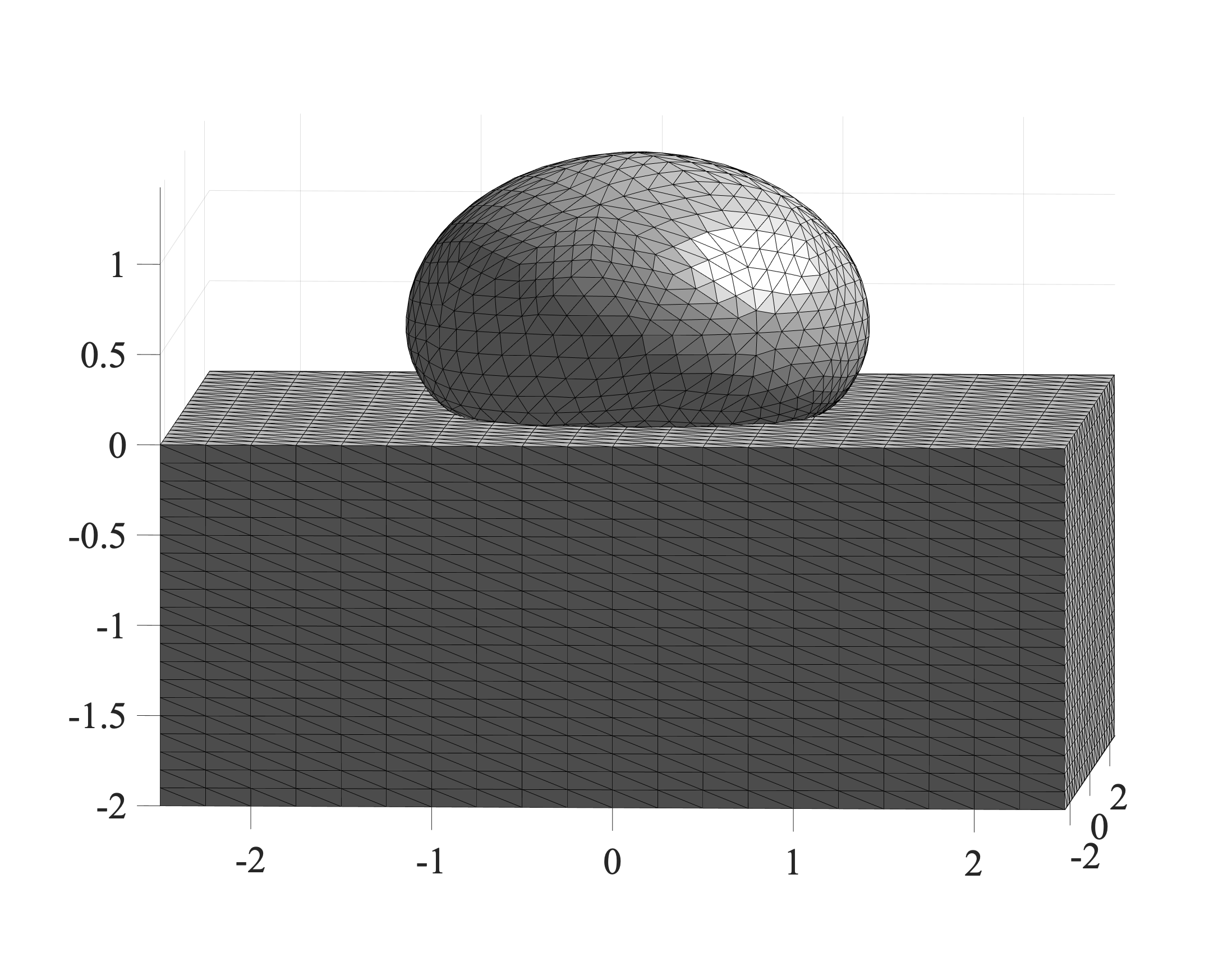}
    \caption{Deformations with zero stiffness of the membrane.}
 \label{fig:zero}
\end{figure}
\begin{figure}
  \centering
\includegraphics[width=5in]{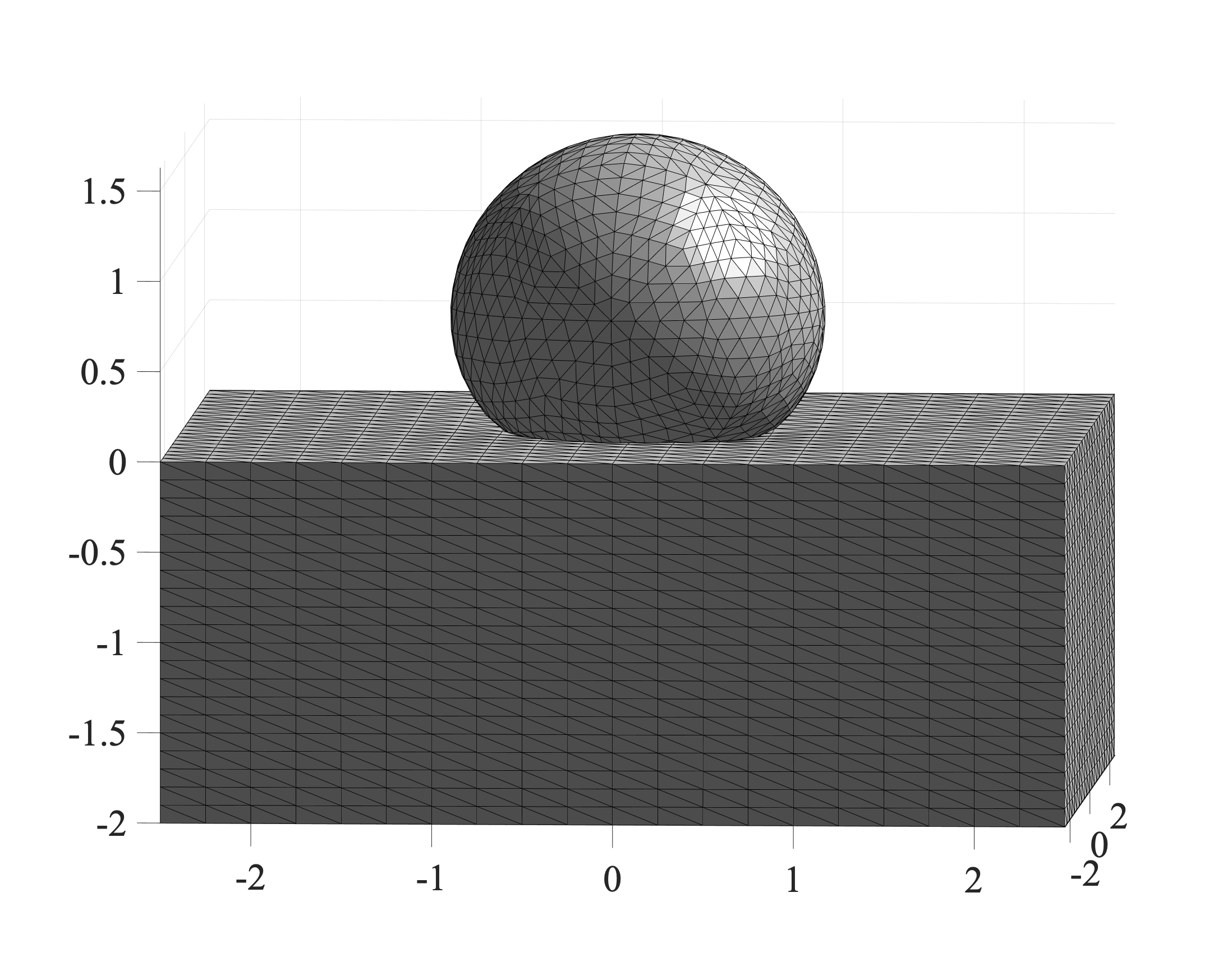}
    \caption{Deformations with $E_\Gamma = 2000$.}
 \label{fig:2000}
\end{figure}
\begin{figure}
  \centering
\includegraphics[width=5in]{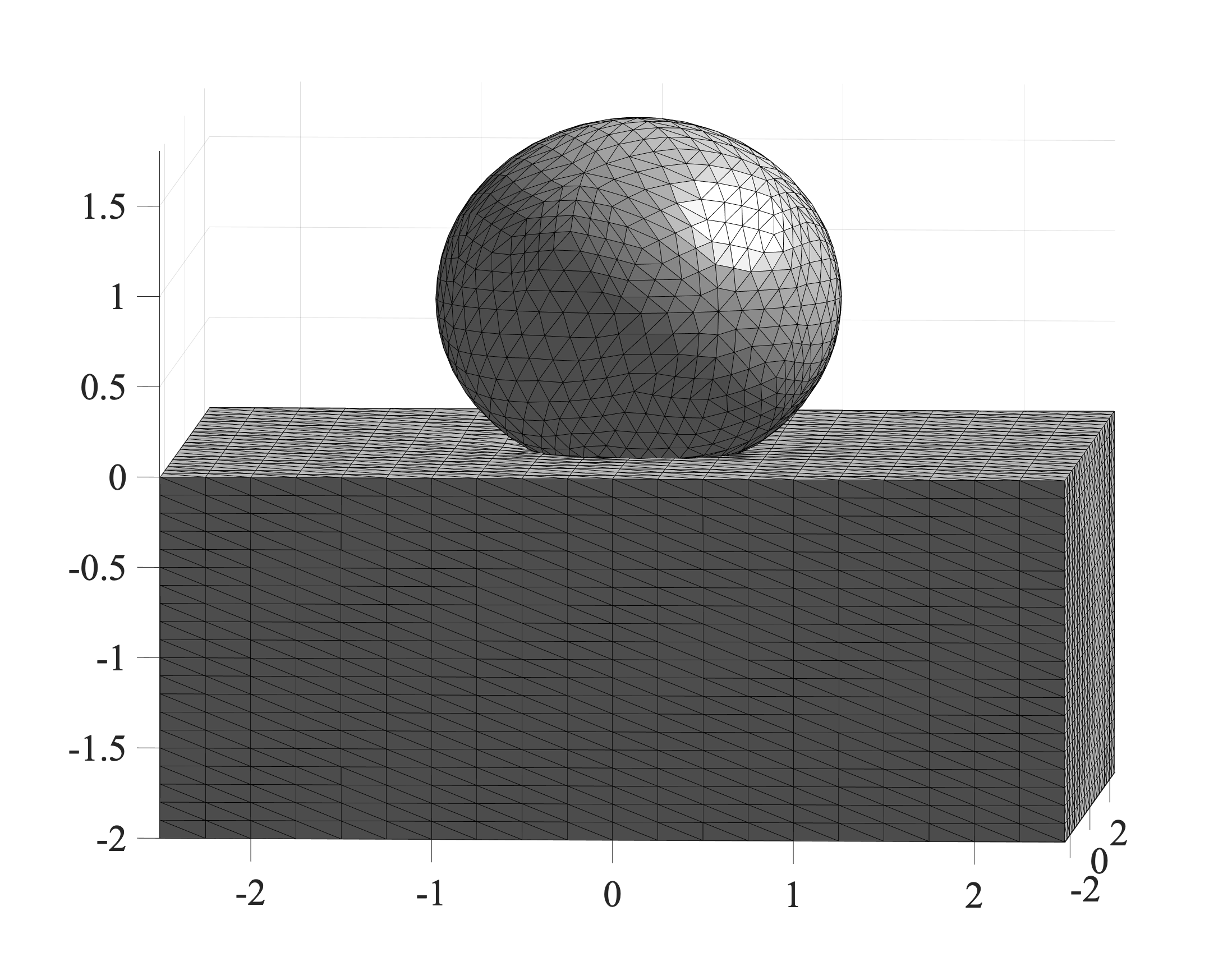}
    \caption{Deformations with $E_\Gamma = 20000$.}
 \label{fig:20000}
\end{figure}

\section{Conclusions}
\label{sec:conclusions}

In this work, we have developed a novel augmented Lagrangian framework for modeling friction-free contact between two elastic bodies. The core of our approach is a Nitsche-based method with a hybrid displacement variable defined on an interstitial layer. This formulation introduces significant flexibility by decoupling the computational domains, allowing the bodies in contact to interact exclusively through the layer. The independent approximation of the layer avoids challenges associated with the intersection of unrelated meshes and opens possibilities for additional modeling, such as incorporating a membrane or other interface effects.

We demonstrated the stability and accuracy of the method by proving stability estimates and deriving error bounds. These theoretical results underline the robustness and convergence of the approach, making it suitable for complex contact problems where traditional methods may face limitations.

Moreover, the hybrid variable approach offers a natural framework for incorporating additional physical phenomena or constraints at the interface, such as thin structures or surface layers, without altering the underlying contact model. This adaptability enhances the potential applicability of the method to a wide range of problems in computational mechanics.

The numerical examples illustrate the proposed method's practical performance, showcasing its ability to handle challenging scenarios with minimal mesh constraints and good agreement with theoretical predictions. These examples highlight the method's versatility in different interface approximations and configurations.

Future work will focus on extending the proposed approach to contact problems involving friction, nonlinear materials, or dynamic interactions. Additionally, exploring further applications of the hybrid interface variable to other types of coupled multiphysics problems, such as fluid-structure interaction or thermal contact, could provide valuable insights and broaden the scope of this method.

In summary, this paper's augmented Lagrangian hybrid Nitsche method offers a flexible, stable, and computationally efficient approach to friction-free contact problems. We believe this framework represents a significant step forward in addressing the challenges associated with contact mechanics and provides a solid foundation for further research and application.

\appendix
\section{Additional Verifications}
\label{sec:kkt}
In this appendix, we include technical details for the reader's convenience.
\begin{itemize}
%

\item  {\bf (\ref{eq:rocka}).} For $a,b\in \IR$, 
\begin{align}
a \leq 0, \quad b \leq 0, \quad ab = 0 \quad \Longleftrightarrow \quad
a = [a-b]_-
\end{align}
\begin{proof} 1. Assume that $a = [a-b]_-$ then we directly have $a\leq 0$. If $a=0$ then $ 0 = [0-b]_- = [b]_+$, which means that $b \leq 0$ and clearly $ab = 0$. If $a<0$ then $a-b<0$ and therefore $a=[a-b]_- = a-b$ which imply $b=0$ and therefore $ab = 0$. 

2. Assume $a\leq 0$, $b\leq 0$, and $ab = 0$. Then $ab=0$ imply either $a$ or $b$ is $0$. If $a=0$ and $b\leq 0$, then 
$[a-b]_- = [-b]_- = [b]_+ = 0$, and it follow that $a = [a-b]_-$. If $b=0$ and $a \leq 0$ then 
$a = [a-b]_- = [a]_- a$.
\end{proof}

\item {\bf (\ref{eq:b-coer}).} For an affine mapping $B:\IR^n \rightarrow \IR^m$ it holds
\begin{equation}
\| [B(\bfv)]_- - [B(\bfw)]_- \|^2_{\IR^m}
\leq
([B(\bfv)]_- - [B(\bfw)]_-, DB(\bfv - \bfw) )_{\IR^m}
\end{equation}
\begin{proof} Note that, for $a,b\in \IR$, 
\begin{align}
([a]_- - [b]_-)(a-b)
&=
[a]_- a - [a]_- b - [b]_- a + b [b]_-
\\
&\geq [a]_-^2 - 2 [a]_-[b]_- + [b]_-^2
\\
&= ([a]_- - [b]_-)^2
\end{align}
where we used the identity $[a]_- a = [a]_-^2$, and the inequality $[a]_- b \leq 
[a]_- [b]_-$, for $a,b \in \IR$.

Next for an affine mapping we have 
\begin{equation}
DB(\bfv-\bfw) = B(\bfv) - B(\bfw)
\end{equation}
and thus
\begin{align}
&([B(\bfv)]_- - [B(\bfw)]_-, DB(\bfv - \bfw) )_{\IR^m} 
\\
&\qquad=
([B(\bfv)]_- - [B(\bfw)]_-, B(\bfv) - B(\bfw) )_{\IR^m}
\end{align}
Proceeding componentwise as above, we will directly obtain the estimate.
\end{proof}

\end{itemize}


\bigskip
\paragraph{Acknowledgement.} This research was supported in part by the Swedish Research
Council Grants Nos.\ 2021-04925, 2022-03908, and the Swedish
Research Programme Essence. EB acknowledges funding from EP/T033126/1 and EP/V050400/1.

\bibliographystyle{abbrv}
\footnotesize{
}

\bigskip
\bigskip
\noindent
\footnotesize {\bf Authors' addresses:}

\smallskip
\noindent
Erik Burman,  \quad \hfill Mathematics, University College London, UK\\
{\tt e.burman@ucl.ac.uk}

\smallskip
\noindent
Peter Hansbo,  \quad \hfill Mechanical Engineering, J\"onk\"oping University, Sweden\\
{\tt peter.hansbo@ju.se}

\smallskip
\noindent
Mats G. Larson,  \quad \hfill Mathematics and Mathematical Statistics, Ume\aa University, Sweden\\
{\tt mats.larson@umu.se}

\end{document}